\documentclass{article}
\title{Some Three- and Four-Dimensional Invariants of Satellite Knots with $(1,1)$ Trefoil Patterns}
\author{Holt Bodish}
\date{}
\usepackage{amsmath,amsthm, amssymb,amsfonts}
\usepackage[alphabetic]{amsrefs}
\usepackage{graphicx}
\usepackage{mathtools}
\usepackage{mathrsfs}
\usepackage{tikz-cd} 
\usepackage[margin=1.5in]{geometry}

\theoremstyle{plain}
\newtheorem{theorem}{Theorem}[section]
\newtheorem{lemma}[theorem]{Lemma}
\newtheorem{proposition}[theorem]{Proposition}
\newtheorem{definition}[theorem]{Definition}
\newtheorem{corollary}[theorem]{Corollary}

\newtheorem{remark}[theorem]{Remark}

\newcommand{\HFhat}{\widehat{\mathrm{HF}}}
\newcommand{\CFAhat}{\widehat{\mathrm{CFA}}}
\newcommand{\txi}{\tilde{\xi}}
\newcommand{\R}{\mathbb{R}}

\newcommand{\Z}{\mathbb{Z}}

\newcommand{\F}{\mathbb{F}}

\newcommand{\CFDhat}{\widehat{\mathrm{CFD}}}
\newcommand{\HFKhat}{\widehat{\mathrm{HFK}}}

\newcommand{\teta}{\tilde{\eta}}

\newcommand{\CFKhat}{\widehat{\mathrm{CFK}}}
\newcommand{\rk}{\text{rk}}

\newcommand{\CFK}{\mathrm{CFK}}
\newcommand{\tikzcirc}[2][black,fill=white]{\tikz[baseline=-0.5ex]\draw[#1,radius=#2] (0,0) circle ;}
\newcommand{\tikzcircle}[2][black,fill=black]{\tikz[baseline=-0.5ex]\draw[#1,radius=#2] (0,0) circle ;}
\newcommand{\tikzcirclee}[2][black,fill=green]{\tikz[baseline=-0.5ex]\draw[#1,radius=#2] (0,0) circle ;}
\newcommand{\talpha}{\tilde{\alpha}}
\newcommand{\tbeta}{\tilde{\beta}}
\begin{document}
\maketitle

\begin{abstract}
We use bordered Floer homology, specifically the immersed curve interpretation of the bordered pairing theorem, to compute various three- and four-dimensional invariants of satellite knots with arbitrary companions and patterns from a family of knots in the solid torus that have the knot type of the trefoil in $S^3$. We compute the three-genus, and bound the four-genus of these satellites. We show that all patterns in this family are fibered in the solid torus. This implies that satellites with fibered companions and patterns from this family are also fibered. We also show that satellites with thin fibered companions or companions $K$ with $\tau(K)=\pm g(K)$ formed from these patterns have left or right veering monodromy. We then use this to show that satellites with thin fibered companion knots $K$ so that $|\tau(K)|<g(K)$ formed from these patterns do not have thin knot Floer homology. 
\end{abstract}

\section{Introduction}
\emph{Knot Floer} homology, introduced by Rasmussen \cite{Rasknot} and Oszv\'{a}th and Szab\'{o} \cite{OSknot}, is an invariant of null homologous knots in the three sphere. Its simplest instantiation takes the form of a bigraded Abelian group, $\HFKhat(S^3,K)\cong \bigoplus_{m,A} \HFKhat_m(S^3,K,A)$. Here $m$ is called the \emph{Maslov grading} and $A$ is called the \emph{Alexander grading}. Knot Floer homology contains information about the knot $K$ and its complement $S^3\setminus \nu(K)$. For example, it detects the three-genus \cite{genusdetection} and fiberedness of the knot \cite{surfacedecomp, Nifibered}, contains information about the monodromy of fibered knots \cite{Niveering}, bounds the number of disjoint, non-isotopic Seifert surfaces in the knot complement \cite{Seifertsurface}, and bounds the four-genus of the knot \cite{OStau}. In this note, we use these detection properties to investigate three- and four-dimensional invariants of satellite knots formed from a family of $(1,1)$-patterns. 

Recall that, given a knot $K \subset S^3$ and a pattern $P \subset S^1\times D^2$, we can construct a new knot, called the ($0$-twisted) satellite knot with \emph{companion} knot $K$ and \emph{pattern} knot $P$, denoted $P(K)$, by removing a tubular neighborhod of $K$ and gluing in the pair $(S^1\times D^2,P)$ so that $S^1\times \{pt\}$ is identified with the Seifert longitude of $K$. A pattern knot $P$ is a $(1,1)$-pattern if it admits a genus-$1$ doubly-pointed bordered Heegaard diagram, a concept that we recall in section \ref{CFA(1,1)}. 

Our main reason for restricting to $(1,1)$-patterns is computational. For an arbitrary pattern $P$, the bordered pairing theorem of \cite{LOT} expresses $\HFKhat(S^3,P(K))$ in terms two invariants: $\CFAhat(S^1\times D^2,P)$ and $\CFDhat(S^3\setminus\nu(K))$. For $(1,1)$-patterns, the work of Chen in \cite{Chen} recasts this pairing theorem in terms of Lagrangian intersection Floer homology of two curves in the punctured torus. This facilitates computation in two ways: it allows one to vary the pattern within a family and it allows one to compute the decompositon into Alexander gradings much more efficiently than with the language of the original bordered pairing theorem.

Many of the computations of knot Floer homology of satellite knots that exist in the literature involve $(1,1)$-patterns. For example the cabling patterns studied in \cite{Hom} (see also \cite{cabling}), Mazur pattens studied in \cite{Levinemazur} and \cite{petkovamazur}, and Whitehead double patterns studied in \cite{Heddenwhitehead} are all $(1,1)$-patterns. Given this, it is interesting to compute knot Floer homology of satellites where the pattern comes from a family of $(1,1)$-patterns. In \cite{Chen} this project is taken up and he examines the case where $P$ is an arbitrary $(1,1)$-pattern $P$ so that $P(U)\simeq U$, called an \emph{unknot pattern}, and the companion knot is the right or left handed trefoil. 

In the following, we use the immersed curve pairing theorem as stated in \cite{Chen} to compute the knot Floer homology of satellites with arbitrary companion knots $K$ and patterns $P$ from a specific family of $(1,1)$-patterns with the property that $P(U) \simeq T_{2,3}$. We will refer to such patterns as \emph{trefoil patterns}. In section \ref{trefoilpatterns} we introduce, for each $p>1$, a trefoil pattern denoted $P_{p,1}$ which is closely related to the $(p,1)$ unknot cabling pattern. Our goal is to investigate various three- and four-dimensional properties of the satellite knots obtained from these trefoil patterns. First, for each $p>1$ and for any knot $K$, we compute the invariant $\tau(P_{p,1}(K))$, an integer valued concordance invariant derived from the knot Floer homology package first defined by \cite{OStau}, in terms of $\tau(K)$ and $\epsilon(K)$.

\begin{theorem}\label{tauoftrefoil}
For the patterns $P_{p,1}$ and for an arbitrary companion knot $K \subset S^3$, we have

\begin{itemize}
\item If $\epsilon(K)=1$, then $\tau(P_{p,1}(K))=(p+1)\tau(K)+1$

\item If $\epsilon(K)=-1$, then $\tau(P_{p,1}(K))=(p+1)(\tau(K)+1)$

\item If $\epsilon(K)=0$, so $\tau(K)=0$, then $\tau(P_{p,1}(K))=\tau(T_{2,3})=1$. 

\end{itemize}
\end{theorem}

\begin{remark}
In forthcoming work with Subhankar Dey, we show that for any knot $K$ with $\epsilon(K)=0$ and for any $(1,1)$-pattern $P$, we have $\epsilon(P(U))=\epsilon(P(K))$ and $\tau(P(U))=\tau(P(K))$, generalizing the third bullet point above.

\end{remark}

As shown in \cite{OStau}*{Corollary 1.3}, the integer $\tau(K)$ satisfies $|\tau(K)| \leq g_4(K)$, where $g_4(K)$ is the smooth four-genus of a knot (the minimal genus of a surface properly embedded in $B^4$ with boundary $K \subset S^3$). This gives the following corollary concerning the slice genus of these satellite knots.
\begin{corollary}
For any companion knot $K$ with $\tau(K)\neq -1$ and  $\epsilon(K) \neq -1$, the satellite knots $P_{p,1}(K)$ are not slice. 
\end{corollary}






Given a pattern in the solid torus, we can associate to it an integer $w(P)$, called the winding number of the pattern, by computing the algebraic intersection between the pattern $P$ and a meridional disk $\{0\}\times D^2$. Given a pattern $P$ with winding number $r$, we define a \emph{relative Seifert surface} for $P$ to be a surface $\tilde{\Sigma}$ in $S^1\times D^2$ so that the interior of $\tilde{\Sigma}$ is disjoint from $P$, and the boundary of $\tilde{\Sigma}$ consists of $P$ together with $r$ coherently oriented longitudes. A pattern is \emph{fibered} if the complement $S^1 \times D^2 \setminus \nu(P)$ is fibered over $S^1$ with fiber surface a relative Seifert surface for $P$. Furthermore, the genus of a pattern, $g(P)$, is defined to be the minimal genus of a relative Seifert surface for $P$. 

For a satellite knot $P(K)$ with a \emph{non-trivial} companion $K$ a result of Schubert \cite{schubert} shows that the three-genus of the satellite knot $g(P(K))$ can be expressed in terms of $w(P)$, $g(K)$ and $g(P)$:
\begin{equation}\label{schubertgenus}
g(P(K))=|w(P)|g(K)+g(P).
\end{equation}

This has the consequence that for any non-trivial knot $K$, the value of $g(P(K))$ is determined by the value of $g(K)$ and $g(P)$. However, $g(P)$ depends only on the pattern. Hence, we can compute $g(P)$ if we can compute the three genus of some satellite with non-trivial companion $K$ and pattern $P$, for example $P(T_{2,3})$. Using the fact that knot Floer homology detects the genus of knots, we prove

\begin{lemma}\label{patterngenus}
For any $p>1$, the trefoil patterns $P_{p,1}$ have $g(P_{p,1})=1$.
\end{lemma}

Now, given the value of $g(P)$, we can determine $g(P(K))$ in terms of $g(K)$ for any non-trivial companion knot $K$ by using equation (\ref{schubertgenus}). This gives the following corollary. Note that the case $K=U$ follows since $g(U)=0$ and $P_{p,1}(U)\simeq T_{2,3}$ has genus $1$.

\begin{corollary}\label{satellitegenus}
For any knot $K$ and for any $p>1$, $g(P_{p,1}(K))=(p+1)g(K)+1$.
\end{corollary}

\begin{remark}\label{remark1}
There are trefoil patterns $P$ with $g(P)>1$ and there are unknot patterns $P$ with $g(P)>0$. The author does not know any upper bound on the genus of a pattern in the solid torus with a fixed knot type in $S^3$, but also does not know examples of patterns of a fixed knot type where the genus gets arbitrarily large. 
\end{remark}

In a similar vein, recall that Hirasawa, Murasugi, and Silver proved in \cite{fiberedsatellite} that a satellite knot with non-trivial companion is fibered if and only if both the pattern and the companion knot are fibered. This has the consequence that to determine if a pattern $P$ is fibered in the solid torus, it is enough to determine if the knot $P(T_{2,3})$ is fibered. Since knot Floer homology detects when a knot is fibered, to show that the pattern $P$ is fibered, it is then enough to compute $\HFKhat(S^3,P(T_{2,3}),g(P(T_{2,3})))$ and show that it has rank $1$.

\begin{theorem}\label{fiberedpattern}
For $p>1$ the pattern knot $P_{p,1}$ is fibered in the solid torus.
\end{theorem}









\begin{remark}
Similar to remark \ref{remark1}, note that there are non-fibered trefoil patterns, and non-fibered unknot patterns. For example, any winding number zero pattern with any knot type in $S^3$ cannot be fibered \cite{fiberedsatellite}.
\end{remark}

One motivation to understand fibered patterns is the result of Ni \cite{Nisutured}*{Theorem 1.2} that the knot Floer homology of satellites with fibered patterns in the top Alexander grading has the same dimension as the knot Floer homology of the companion in the top Alexander grading. That is 
\begin{equation}\label{nifiberedpattern}
\rk\HFKhat(S^3,K,g(K))=\rk\HFKhat(S^3,P(K),g(P(K))).
\end{equation}

This theorem, when combined with the work of Juhasz in \cite{surfacedecomp,Seifertsurface} which relates the knot Floer homology in the top Alexander grading to the sutured Floer homology of the complement of a Seifert surface for the knot $K$ has the following consequences. 

\begin{proposition}\label{Seifertsurface}
If $K$ is a knot with $\rk(\HFKhat(S^3,K,g(K)))<4$ and $P$ is a fibered pattern, then for all $i\geq1$ the knots $K$ and $P^i(K)$ have unique minimal genus Seifert surfaces. 
\end{proposition}



\begin{proposition}\label{taut}
If $K$ is a knot with $\rk(\HFKhat(S^3,K,g(K)))=3$ and $P$ is a fibered pattern, then $K$ and $P^i(K)$ admit depth $\leq 1$ taut foliations transverse to the boundary.
\end{proposition}
Recall that fibered knots have unique minimal genus Seifert surfaces. These propositions can be viewed as generalizations of this fact. In particular, by Theorem \ref{fiberedpattern}, these propositions apply to the patterns $P_{p,1}$.

Finally we study the next to top Alexander graded piece of the knot Floer homology of these satellite knots. In the case that $K$ is a fibered knot, $\HFKhat(S^3,K,g(K)-1)$ contains information about the monodromy of the fibration, in the following sense. 

\begin{theorem}[\cite{Niveering}]\label{veering} If $K$ is a fibered knot and $\rk(\HFKhat(S^3,K,g(K)-1))=1$, then the monodromy of $K$ is either left or right veering. 
\end{theorem}

\begin{remark} 
There is no analogue of equation (\ref{nifiberedpattern}) for the next to top Alexander graded piece of knot Floer homology of a satellite and its companion. In general, there is not even an inequality relating them, even for fibered patterns. For example  $\rk(\HFKhat(S^3,T_{2,3},0)=1$ and $\rk\HFKhat(S^3,(T_{2,3})_{2,1},1)=2$ and as Theorem \ref{monodromy} shows, constructing satellites with certain patterns can decrease the rank in the next to top Alexander graded piece by an arbitrary amount. Note, certain families of patterns do preserve the property of having one dimensional Floer homology in the next to top Alexander grading, for example if $K$ is an L-space knot and $P$ is a pattern so that $P(K)$ is also an L-space knot (for example the $(p,q)$ cable pattern with $\frac{q}{p}\geq 2g(K)-1$) then by \cite{geography} both $K$ and $P(K)$ have one dimensional Floer homology in the next to top Alexander grading. 
\end{remark}

Recall that the $\delta$-grading on knot Floer homology is define by $\delta=m-A$. We call a knot $K$ \emph{Floer thin} (or \emph{thin}) if the $\delta$-grading is constant for all generators of $\HFKhat(S^3,K)$.

\begin{theorem}\label{monodromy}

For each $p>1$, and for any fibered knot $K$ with $\tau(K)=\pm g(K)$, or for any fibered thin knot $K$, we have

$$\rk(\HFKhat(S^3,P_{p,1}(K),g(P_{p,1}(K))-1))=1.$$

\end{theorem}

\begin{corollary}
For any fibered knot $K$ with $\tau(K)=\pm g(K)$, or for any fibered thin knot $K$, the fibered knot $P_{p,1}(K)$ has left or right veering monodromy. 

\end{corollary}

Lastly, we use Theorem \ref{monodromy} to show that for some fibered companion knots $K$, the satellite knots $P_{p,1}(K)$ are not Floer thin. The main result we use is \cite{baldwinveering}*{Corollary 1.7} which says that a fibered thin knot with $|\tau(K)|<g(K)$ cannot have left or right veering monodromy.

\begin{proposition}\label{thinsatellite}
If $K$ is a non-trivial fibered knot with thin knot Floer homology such that $|\tau(K)|<g(K)$, then the knot Floer homology of $P_{p,1}(K)$ is not thin. 
\end{proposition}

Since quasialternating knots have thin knot Floer homology by \cite{quasialt}, we have the following consequence of Proposition \ref{thinsatellite}.
\begin{corollary}
For any $p>1$ and for any thin fibered knot $K$ with $|\tau(K)|<g(K)$, the knots $P_{p,1}(K)$ are not quasialternating. 

\end{corollary}

\subsection*{Organization} 
In section \ref{bordered} we introduced the bordered pairing theorem from \cite{LOT} and recall the work of \cite{HRW} reinterpreting the bordered invariants in terms of immersed curves in the punctured torus. In section \ref{(1,1)pairingtheorem}, we recall Chen's immersed curve version of the pairing theorem from \cite{Chen}. In section \ref{trefoilpatterns} we prove Theorem \ref{tauoftrefoil}. In section \ref{threegenusandfibered}, we prove Theorem \ref{fiberedpattern}, as well as propositions \ref{Seifertsurface} and \ref{taut}. In section \ref{nexttotop}, we prove Theorem \ref{monodromy} and Proposition \ref{thinsatellite}.

\subsection*{Acknowledgements}
The author thanks Subhankar Dey for helpful discussions relating to satellites with $(1,1)$ patterns, and his advisor Robert Lipshitz for helpful conversations.

\section{Bordered Floer Homology}\label{bordered}

In this section we introduced the necessary notation to state and interpret the pairing theorem for bordered Floer homology of \cite{LOT}. Bordered Floer homology is an invariant that is used to study Heegaard Floer homology of three manifolds that have been decomposed along essential embedded surfaces. In our case, studying satellite operators, we are interested in decomposing the ambient three manifold, $S^3$ together with a knot $K$, along an essential torus. Then one can compute certain algebraic invariants of both sides and the Floer homology of the ambient three manifold (together with the knot filtration) can be computed by suitably combining these invariants.

In \cite{LOT}, Lipshitz, Oszv\'{a}th and Thurston associate, to a three manifold with parameterized torus boundary, a type A and D structure over the torus algebra $\mathcal{A}$. We now briefly describe these concepts. The torus algebra $\mathcal{A}$ is defined as follows. Over $\F$ it has a basis  consisting of two mutually orthogonal idempotents $\iota_0$ and $\iota_1$ and six other nontrivial elements $\rho_1,\rho_2,\rho_3,\rho_{12},\rho_{23},\rho_{123}$. The non-zero products in the algebra are given as follows:

$$\rho_1\rho_2=\rho_{12} \quad \rho_2\rho_3=\rho_{23} \quad \rho_1\rho_{23}=\rho_{12}\rho_3=\rho_{123}$$

\begin{align*}
\rho_1&=\iota_0\rho_1\iota_1 \qquad& \rho_2&=\iota_1\rho_2\iota_0 &\qquad \rho_3&=\iota_0\rho_3\iota_1\\
\rho_{12}&=\iota_0\rho_{12}\iota_0 \qquad &\rho_{23}&=\iota_1\rho_{23}\iota_1 &\qquad \rho_{123}&=\iota_0\rho_{123}\iota_1\\
\end{align*}
 If we let $\mathcal{I} \subset \mathcal{A}$ denote the subring of idempotents, then a type $D$ structure over $\mathcal{A}$ is a unital left $\mathcal{I}$ module $N$ together with an $\mathcal{I}$ linear map $\delta: N \to \mathcal{A} \otimes_{\mathcal{I}} N$ such that

$$(\mu \otimes \mathbb{I})\circ (\mathbb{I} \otimes \delta) \circ \delta =0$$

A type A structure is a right unital $\mathcal{I}$ module $M$ with a collection of maps $m_{i+1}:M \otimes \mathcal{A}^i \to M$, for $i \geq 0$ such that

\begin{equation}\label{inftyrelation}
0=\sum_{i=1}^n m_{n-i}(m_i(x\otimes a_1\otimes \cdots \otimes a_{i-1})\otimes\cdots \otimes a_{n-1})+\sum_{i=1}^{n-2}m_{n-1}(x\otimes \cdots\otimes a_ia_{i+1}\otimes \cdots \otimes a_n)\end{equation} and so that

$$m_2(x,1)=x$$
$$m_i(x,\cdots,1,\cdots)=0$$


Given a type A structure $M$ and a type D structure $N$, we can form a chain complex, called as the box tensor product and denoted $M \boxtimes N$. The underlying vector space is the tensor product $M \otimes_{\mathcal{I}} N$, and the differential is defined by

\begin{equation}
\partial^{\boxtimes}(x\otimes y)=\sum_{i=0}^{\infty} (m_{i+1}\otimes \mathbb{I})(x\otimes \delta_i(y))
\end{equation}
In the case that the type D structure is bounded, as defined in in \cite{LOT}*{section 2}, then the above sum is finite and the box tensor complex is well defined.

In what follows, we are interested in the following version of the bordered pairing theorem.

\begin{theorem}[\cite{LOT}*{Theorem 11.19}]\label{lotpairing}
Suppose $Y$ is a closed $3$-manifold decomposed as $Y=Y_1\cup Y_2$ with $\partial Y_1 \cong -\partial Y_2\cong  T^2$. Suppose further that $K \subset Y_1$ is a knot with is null homologous in $Y$. Then up to homotopy equivalence of chain complexes 

$$g\CFKhat(Y,K) \simeq \CFAhat(Y_1,K) \boxtimes \CFDhat(Y_2)$$
\end{theorem}

We will give the immersed curve interpretation of this pairing theorem due to \cite{Chen} for $(1,1)$ patterns in section \ref{(1,1)pairingtheorem}. First, we will describe in more detail how to compute and interpret $\CFDhat(S^3\setminus \nu(K))$ and $\CFAhat(S^1\times D^2,P)$ as immeresed curves in the punctured torus in the next two sections. 

\subsection{$\CFDhat(S^3\setminus \nu(K))$ from $\CFK^{-}(K)$}\label{CFDfromCFK}
In this section, we recall the algorithm from \cite{LOT}*{Section 11.5} for computing $\CFDhat(S^3\setminus \nu(K))$ from $\CFK^{-}(K)$. For the definitions of reduced, filtered basis, we refer the reader to the original source (see also \cite{geography}).
We call a filtered reduced basis over $\F[U]$ vertically simplified if for each basis element $x_i$ exactly one of the following conditions is satisfied

\begin{itemize}

\item There is a unique incoming vertical arrow, and no outgoing vertical arrow, or
\item There is a unique outgoing vertical arrow and no incoming vertical arrow, or
\item There are no vertical arrows. 

\end{itemize}

A horizontally simplified basis is defined similarly, replacing vertical by horizontal in the above. Given a knot $K$ and a framing $n$, exists a pair of bases $\teta=\{\teta_1,\dots,\teta_{2k}\}$ and $\txi=\{\txi_1,\dots,\txi_{2k}\}$ for $CFK^{-}(K)$ that are horizontally and vertically simplified respectively. They are indexed so that for every pair $\teta_{2i-1}$ and $\teta_{2i}$ there is a horizontal arrow of length $l_i\geq 1$ connecting them and similarly, there is a vertical arrow of length $k_i\geq 1$ from $\txi_{2i-1}$ to $\txi_{2i}$. There are corresponding bases $\xi=\{\xi_0,\dots,\xi_{2k}\}$ and $\eta=\{\eta_0,\dots,\eta_{2k}\}$ for $\iota_0\CFDhat(X_k,n)$ so that if $\txi_j=\sum_{i=0}^{2k} a_{ij}\teta_i$ and $\eta_j=\sum_{i=0}^{2k}b_{ij}\txi_i$, then the corresponding change of bases formulas hold with the coefficients restricted to $U=0$. The summand $\iota_1\CFDhat$ has basis $$\bigcup_{i=1}^k\{\kappa_1^i,\dots,\kappa_{k_i}^i\}\cup\bigcup_{i=1}^k \{\lambda_1^i,\dots,\lambda_{l_i}^i\} \cup\{\mu_1,\dots,\mu_{|2\tau(K)-n|}\}$$

There are non-zero coefficient maps induced from the horizontal and vertical arrows in the complex for $CFK^{-}$ as follows. A length $k_i$ vertical arrow from $\xi_{2i-1}$ to $\xi_{2i}$ induces coefficient maps

$$\xi_{2i-1} \xrightarrow{\rho_1} \kappa_1^i \xleftarrow{\rho_{23}} \kappa_2^{i} \dots\xleftarrow{\rho_{23}}\kappa^i_{k_i}\xleftarrow{\rho_{123}}\xi_{2i}$$

Similarly, for each length $l_i$ horizontal arrow from $\eta_{2i-1}$ to $\eta_{2i}$, we get coefficient maps

$$\eta_{2i-1} \xrightarrow{\rho_3} \lambda_1^i \xrightarrow{\rho_{23}} \lambda^i_{2}\xrightarrow{\rho_{23}} \dots \xrightarrow{\rho_{23}}\lambda_{l_i}^i \xrightarrow{\rho_2} \eta_{2i}$$

Additionally, there are coefficient maps from $\xi_0$ to $\eta_{0}$ depending on the framing and the value of the invariant $\tau(K)$.

\begin{itemize}

\item $\xi_0 \xrightarrow{\rho_{12}} \eta_0 \quad \text{if} \quad n=2\tau(K)$

\item $\xi_0 \xrightarrow{\rho_1} \mu_1 \xleftarrow{\rho_{23}} \dots \xleftarrow{\rho_{23}} \mu_m \xleftarrow{\rho_3} \eta_0  \quad \text{if} \quad n<2\tau(K) \quad m=2\tau(K)-n$

\item $\xi_0 \xrightarrow{\rho_{123}} \mu_1 \xrightarrow{\rho_{23}} \dots \xrightarrow{\rho_{23}} \mu_m \xrightarrow{\rho_2} \eta_0  \quad \text{if} \quad n>2\tau(K), \quad m=n-2\tau(K)$

\end{itemize}

For example, for the knot $K=T_{2,3}$, the right-handed trefoil, $\CFK^{-}(T_{2,3})$ has a simultaneously vertically and horizontally simplified $\F[U]$ basis $\{\txi_0,\txi_1,\txi_2\}$ with differential given by $\partial(\txi_1)=U\txi_0+\txi_2$. Applying the above algorithm, we get the type $D$ structure shown in figure \ref{trefoil0framed}.

\begin{figure}[!tbp]
\begin{center}
\begin{tikzcd}
\xi_0\arrow{dr}{\rho_1} & & \lambda \arrow{ll}[swap]{\rho_2} & \xi_1 \arrow{l}[swap]{\rho_3}\arrow{dd}{\rho_1}\\
& \mu_1 & &\\
&& \mu_2 \arrow{ul}{\rho_{23}}& \kappa \\
&&& \xi_2 \arrow{u}[swap]{\rho_{123}}\arrow{ul}{\rho_3}
\end{tikzcd}
\caption{Type D structure for $0$-framed right handed trefoil complement}\label{trefoil0framed}
\end{center}
\end{figure}
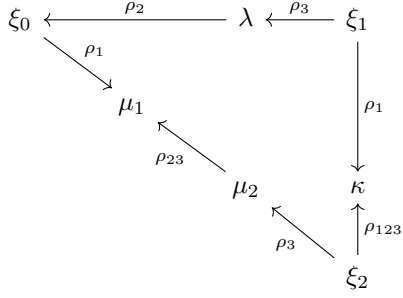

For any knot $K$ in $S^3$, there is always a vertically distinguished element of a horizontally simplified basis, which is an element in a horizontally simplified basis with no incoming or outgoing vertical arrows. Similarly, there is a horizontally distinguished element of a vertically simplified basis. In \cite{Hom}*{Lemma 3.2}, she shows that it is always posible to find a horizontally simplified basis for $\CFK^{\infty}(K)$ so that one of the horizontal basis elements $\xi_0$ is the vertically distinguished generator of some vertically simplified basis. Note that the concordance invariant $\epsilon(K)$ can be defined in terms of the generator $\xi_0$: If $\xi_0$ occurs at the end of a horizontal arrow, then $\epsilon(K)=1$, if $\xi_0$ occurs at the beginning of a horizontal arrow, then $\epsilon(K)=-1$. If there is no horizontal arrow to or from $\xi_0$, then $\epsilon(K)=0$.

\subsection{Immersed Curves for knot complements}\label{immersedD}

Given a type $D$ structure over the torus algebra, like $\CFDhat(S^3\setminus \nu(K))$, the work in \cite{HRW} shows how we can represent it as an immersed multicurve with local systems in the torus, which we now describe. The first step is to construct a decorated graph from the type D structure. Let $N$ be a type $D$ strucutre over the torus algebra, and let $N_i=\iota_iN$. This gives a decomposition $N=N_0\oplus N_1$. Given bases $B_i$ of $N_i$, for $i=0,1$, we construct a decorated graph $\Gamma$ as follows. The vertices of $\Gamma$ are in correspondence with the basis elements and are labelled $\tikzcircle{2pt}$ or $\tikzcirc{2pt}$ depending on if the vertex corresponds to a basis element in $B_0$ or $B_1$ respectively. Suppose now that we have two vertices corresponding to basis elements $x$ and $y$ such that $\delta(x)=\rho_I\otimes y +\cdots$, for $I \in \{\emptyset,1,2,3,12,23,123\}$. In this case we put an edge labelled $\rho_I$ from $x$ to $y$. A decorated graph is called reduced if no edge labelled by $\rho_{\emptyset}$ appears. The next step is to take a decorated graph and turn it into an immersed train track in the punctured torus. Let $T^2=\R^2/\Z^2$ and let $w=(1-\epsilon,1-\epsilon)$ be a basepoint. Let $\mu$ and $\lambda$ be the images of the $x$ and $y$ axes respectively and embed the vertices of $\Gamma$ into $T^2$ so that the $\tikzcircle{2pt}$ vertices lie on $\lambda$ in the interval $\{0\} \times [\frac{1}{4},\frac{3}{4}]$ and the $\tikzcirc{2pt}$ vertices lie on $\lambda$ in the interval $[\frac{1}{4},\frac{3}{4}]\times \{0\}$. Then we embed the edges into the torus according to the rules shown in \cite{HRW}*{Figure 19} (see also figure \ref{typeDtrefoilcurve}). In general this train track is not ncessarily an immersed curve, but work in \cite{HRW} shows that for type $D$ structures that arise from 3-manifolds with torus boundary one can always choose a nice basis so that the train track is an immersed curve (possibly with local systems). For example, we construct the immersed curve associated to the trefoil complement in figure \ref{typeDtrefoilcurve}, where for example the arc from $\xi_1$ to $\kappa$ indicates the presence of a $\rho_1$ edge from $\xi_1$ to $\kappa$ in the decorated graph. We will denote this immersed curve by $\alpha(K)$.


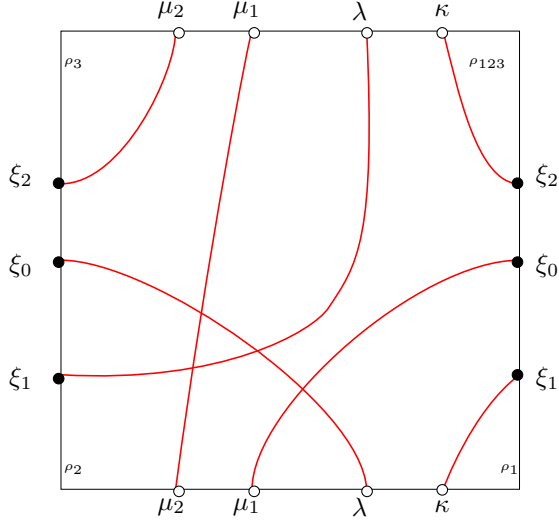
\begin{figure}
\begin{center}
\begin{tikzpicture}

\node[anchor=south west,inner sep=0] at (0,0)    {\includegraphics[scale=.2]{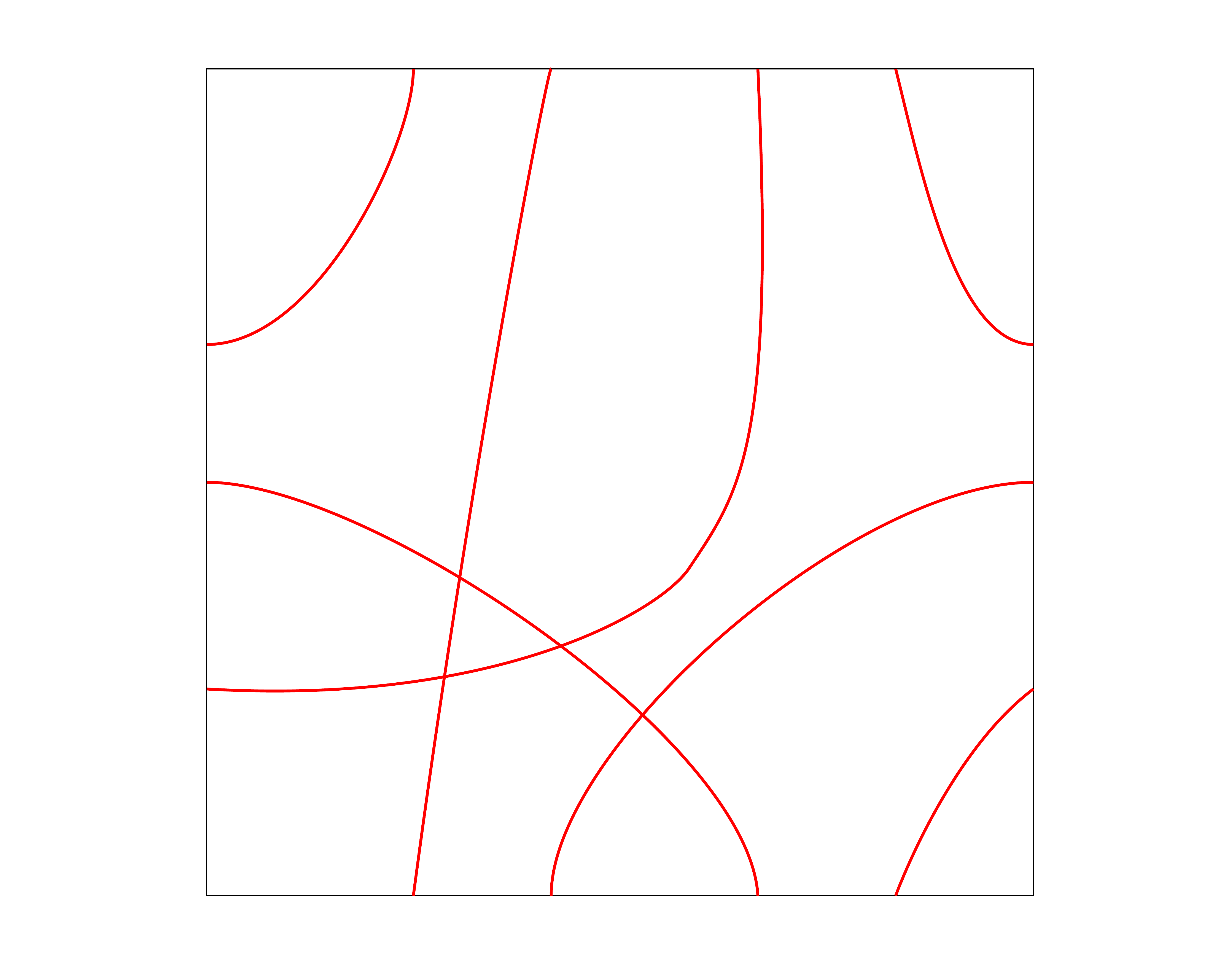}};
\node at (5.5,.4) {$\lambda$};
\node at (5.6,.6) {$\tikzcirc{2pt}$};
\node at (5.5,7) {$\lambda$};
\node at (5.6,6.7) {$\tikzcirc{2pt}$};
\node at (3,.4) {$\mu_2$};
\node at (3.1,.6) {$\tikzcirc{2pt}$};
\node at (3,7) {$\mu_2$};
\node at (3.1,6.7) {$\tikzcirc{2pt}$};
\node at (4,.4) {$\mu_1$};
\node at (4.1,.6) {$\tikzcirc{2pt}$};
\node at (4,7) {$\mu_1$};
\node at (4.1,6.7) {$\tikzcirc{2pt}$};
\node at (6.6,.4) {$\kappa$};
\node at (6.6,.62) {$\tikzcirc{2pt}$};
\node at (6.6,7) {$\kappa$};
\node at (6.6,6.7) {$\tikzcirc{2pt}$};
\node at (1,2.1) {$\xi_1$};
\node at (1.5,2.1) {$\tikzcircle{2pt}$};
\node at (1,3.6) {$\xi_0$};
\node at (1.5,3.65) {$\tikzcircle{2pt}$};
\node at (1,4.8) {$\xi_2$};
\node at (1.5,4.7) {$\tikzcircle{2pt}$};
\node at (8,2.1) {$\xi_1$};
\node at (7.6,2.15) {$\tikzcircle{2pt}$};
\node at (8,3.6) {$\xi_0$};
\node at (7.6,3.65) {$\tikzcircle{2pt}$};
\node at (8,4.8) {$\xi_2$};
\node at (7.6,4.7) {$\tikzcircle{2pt}$};
\node[font=\tiny] at (7.2,6.3) {$\rho_{123}$};
\node[font=\tiny] at (7.5,.9) {$\rho_1$};
\node[font=\tiny] at (1.7,.9) {$\rho_2$};
\node[font=\tiny] at (1.7,6.3) {$\rho_3$};
\end{tikzpicture}
\caption{The immersed curve associated to the $0$ framed trefoil complement}\label{typeDtrefoilcurve}
\end{center}
\end{figure}

\subsubsection{Properties of Immersed Multicurves for Knot Complements}\label{concordance}

In this section we recall how the immersed curve $\alpha(K)$ encodes the concordance invariants $\tau(K)$ and $\epsilon(K)$ as well as the genus of the knot $g(K)$. In order to do this, we fix a representative of the lift of the immersed curve to the universal cover, called the \textit{peg-board representative} of the immersed curve. This is discussed in \cite{HRW}*{Section 4.2}. In brief, we assume that we have chosen a minimal length representative of the immersed multicurve. Given a peg-board representative of $\alpha(K)$, the genus of the knot is half the maximal number of pegs between the minimum and maximum height attained by the immersed curve. The invariants $\tau(K)$ and $\epsilon(K)$ are related to the essential component $\gamma_0$ of the immersed curve, see \cite{HRW} and \cite{cabling}*{Proposition 2}. The essential component $\gamma_0$ is the unique non-vertical segment of the immersed curve, in the sense that all other components are supported in a neighborhood of the meridian, and the component $\gamma_0$ wraps once around the cylinder (in the covering of the torus corresponding to the longitudinal subgroup). As mentioned in \cite{HRW1}*{Remark 50} this component does not carry any non-trivial local system as only one curve component can wrap around the cylinder (since otherwise the meridional filling would have rank $\geq 2$). This observation, together with the discussion surrounding \cite{Hom}*{Lemma 3.2} in Section \ref{CFDfromCFK} implies the following lemma concerning the shape of the essential component of $\alpha(K)$ lifted to the universal cover.

\begin{lemma}\label{essentialcomponent}
Suppose that $K$ is a knot in $S^3$ and that $\gamma_0$ is the essential curve component of $\alpha(K)$ lifted to the universal cover. 
\begin{itemize}
\item If $\epsilon(K)=1$ and $\tau(K)\geq0$ $\gamma_0$ slopes upwards for $2\tau(K)$ rows and turns down at the top and up at the bottom

\item If $\epsilon(K)=-1$ and $\tau(K)\geq0$, then $\gamma_0$ slopes upwards for $2\tau(K)$ rows and turns up at the top and down at the bottom

\item If $\epsilon(K)=1$ and $\tau(K)\leq0$ then $\gamma_0$ slopes downwards for $2\tau(K)$ rows and turns down at the bottom at up at the top

\item If $\epsilon(K)=-1$ and $\tau(K)\leq0$ then $\gamma_0$ slopes downwards for $2\tau(K)$ rows and turns up at the bottom and down at the top. 

\item If $\epsilon(K)=0$, then $\tau(K)=0$ and $\gamma_0$ is horizontal at height $0$. 
\end{itemize}

In each case the remaining portion of the essential component of the immersed curve and any other component of the immersed curve are contained in a neighborhood of the meridian. 
\end{lemma}
\begin{proof}
We will show that the immersed curve has the claimed form in the case that $\tau(K)>0$ and $\epsilon(K)=1$. The rest of the cases are similar. As mentioned above, in \cite{Hom}, Hom constructs a horizontally simplified basis $\{\xi_0,\eta_0,\cdots,\eta_N\}$ so that $\xi_0$ is the distiguished element in a vertically simplified basis with no incoming or outgoing vertical arrows. In the case $\epsilon(K)=1$, this generator appears at the end of a horizontal arrow. Suppose that $\eta_1 \xrightarrow{} \xi_0$ is a length $l$ arrow from $\eta_1$ to $\xi_0$. In this case, the portion of $\CFDhat(S^3\setminus \nu(K))$ has the following form: From the length $l$ horizontal arrow from $\eta_1$ to $\xi_0$, the algorithm in \cite{LOT} produces a sequence of type $D$ operations


$$\eta_{1} \xrightarrow{\rho_3} \lambda_1^1 \xrightarrow{\rho_{23}} \lambda^1_{2}\xrightarrow{\rho_{23}} \dots \lambda_{l}^1 \xrightarrow{\rho_2} \xi_{0}$$

(Note that this part of the type D operations or immersed curve is what changes when $\epsilon(K)$ changes sign)

Since $\tau(K)>0$, the unstable chain takes the form

$$\xi_0 \xrightarrow{\rho_1} \mu_1 \xleftarrow{\rho_{23}}\cdots \xleftarrow{\rho_{23}} \mu_{2\tau(K)}\xleftarrow{\rho_{3}} \eta_0$$

(Note that this part of the type D operations or immersed curve is what changes when $\tau(K)$ changes sign)

Using the procedure described in \cite{HRW1} and the previous section, this decorated graph becomes the portion of the immersed curve shown in figure \ref{tauepsilonposimmersedcurve}. As claimed, the immersed curve slopes upwards for $2\tau(K)$ rows, turns down at the top (from the $\rho_2$ from $\lambda_l^1$ to $\xi_0$) and turns up at the bottom by the symmetry of the immersed curve under the elliptic involution. The remaining bullet points follow similarly.

The fact that the remaining portion of the immersed curve is contained in a neighborhood of the meridian follows since the meridional filling of any knot complement has rank one. If any other component wrapped around the longitude, this would imply that the meridional filling has rank $\geq 2$. \qedhere

\begin{figure}
\begin{center}
\begin{tikzpicture}

\node[anchor=south west,inner sep=0] at (0,0)    {\includegraphics[scale=.2]{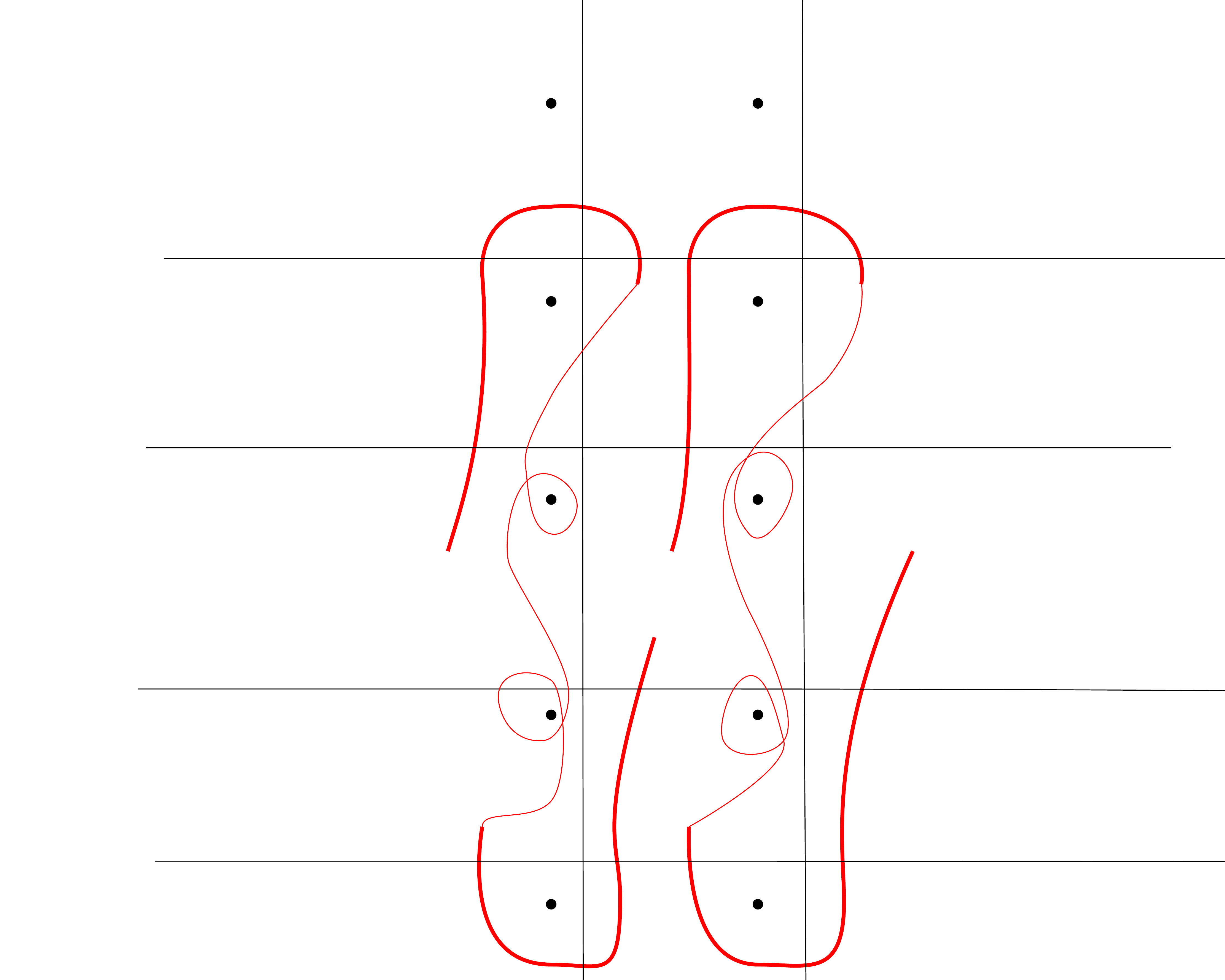}};
\node[font=\tiny] at (4.2,.65) {$w$};
\node[font=\tiny] at (4.2,2.1) {$w$};
\node[font=\tiny] at (4.2,3.7) {$w$};
\node[font=\tiny] at (4.2,5.2) {$w$};
\node[font=\tiny] at (4.2,6.5) {$w$};
\node[font=\tiny] at (5.7,.65) {$w$};
\node[font=\tiny] at (5.7,2.1) {$w$};
\node[font=\tiny] at (5.7,3.7) {$w$};
\node[font=\tiny] at (5.7,5.2) {$w$};
\node[font=\tiny] at (5.7,6.5) {$w$};
\node at (3.55,.9) {$\tikzcirc{2pt}$};
\node[font=\tiny] at (4.8,.9) {$\mu_{m}$};
\node[font=\tiny] at (5.2,2.15) {$\mu_{m-1}$};
\node[font=\tiny] at (4.9,4.8) {$\rho_{23}$};
\node[font=\tiny] at (3.3,.3) {$\rho_{123}$};
\node[font=\tiny] at (4.8,.3) {$\rho_3$};
\node at (4.75,2.15) {$\tikzcirc{2pt}$};
\node[font=\tiny] at (4.55,3) {$\rho_{23}$};
\node at (4.9,3) {$\vdots$};
\node[font=\tiny] at (4.3,1.5) {$\rho_{23}$};
\node at (5.05,3.9) {$\tikzcirc{2pt}$};
\node[font=\tiny] at (5.3,3.9) {$\mu_2$};
\node at (4.55,.9) {$\tikzcirc{2pt}$};
\node at (4.3,.12) {$\tikzcircle{2pt}$};
\node[font=\tiny] at (4.3,.25) {$\eta_0$};
\node at (3.55,5.3) {$\tikzcirc{2pt}$};
\node at (4.3,5.7) {$\tikzcircle{2pt}$};
\node[font=\tiny] at (4.3,6) {$\xi_0$};
\node[font=\tiny] at (3.5,5.6) {$\rho_1$};
\node[font=\tiny] at (3.3,5.3) {$\mu_1$};
\node[font=\tiny] at (4.85,5.7) {$\rho_2$};
\node[font=\tiny] at (5,5.3) {$\lambda_l^1$};
\node at (4.75,5.3) {$\tikzcirc{2pt}$};
\end{tikzpicture}
\caption{The essential component of the immersed curve for a knot $K$ with $\tau(K)>0$ and $\epsilon(K)=1$. The curve crosses at heights $-\tau(K)$ and $\tau(K)$. The lighter portion of the curve indicates that $\gamma_0$ is potentially immersed in the punctured torus, but is contained in a small neighborhood of the meridian, along with all the other components of the immersed multicurve $\alpha(K)$}\label{tauepsilonposimmersedcurve}
\end{center}
\end{figure}
\end{proof}

\subsection{$\CFAhat(S^1\times D^2,P)$ for $(1,1)$-patterns $P\subset S^1\times D^2$}\label{CFA(1,1)}
As we saw in the previous section, the type $D$ structure from the pairing theorem can be obtained algorithmically from knowledge of $\CFK^{-}(K)$. For the type A side, there is no such algorithm for determining $\CFAhat(S^1\times D^2,P)$ in terms of $\CFK^-(P(U))$. However, when the pattern $(S^1\times D^2,P)$ admits a particular type of Heegaard diagram, called a genus-$1$ doubly-pointed bordered Heegaard diagram, we can compute $\CFAhat(S^1\times D^2,P)$ directly. We now describe how to do this. First, we introduce the notation of a genus 1 doubly-pointed bordered Heegaard diagram.

\begin{definition}
A genus-1 doubly-pointed bordered Heegaard diagram is a five tuple $(\Sigma,\alpha^a,\beta,w,z)$. Here $\Sigma$ is a compact oriented surface of genus $1$ with a single boundary component. The alpha arcs $\alpha^a=(\alpha_1^a,\alpha_2^a)$ are a pair of properly embedded, disjoint arcs in $\Sigma$ with a fixed order to the intersections $\alpha^a \cap \partial\Sigma$. The basepoint $w$ lies on the boundary of $\Sigma$ in the complement of the endpoints of the $\alpha$ arcs; i.e. $w\subset\partial\Sigma \setminus \partial \alpha^a$. The resulting subdivision of $\partial \Sigma$ results in the data of a pointed matched circle. The $\beta$-curve is an embedded closed loop in $\Sigma$ so that $\beta$ is transverse to the $\alpha$-arcs and the complement $\Sigma \setminus \beta$ is connected. Furthermore, we place a basepoint $z$ in the interior of $\Sigma$ without the $\alpha$-arcs and $\beta$-circles, so that if we forget the $z$ basepoint, the $\beta$ curve is isotopic to $\alpha^a_2$. 

\end{definition}

This data specifies a three manifold with torus boundary together with a knot. The three manifold and knot can be recovered by the following recipe. Attach a two-handle to $\Sigma \times [0,1]$ along $\beta \times \{1\}$. The knot is specified by connecting the $z$ basepoint to the $w$ basepoint in the complement of $\beta$ and pushing the arc into the handlebody compressed by the $\beta$ curve and connecting $w$ to $z$ in the complement of $\alpha^a$ in $\Sigma$. Note that the $\alpha$-arcs are the cores of the $1$-handles of the boundary torus. 
In our case, we have $\alpha_1^a=\lambda$ and $\alpha_2^a=\mu$ the longitude and meridian of the torus boundary $\partial(S^1\times D^2)$. See figure \ref{P_3doublypointeddiagram} for an example of a genus 1 doubly pointed bordered Heeggard diagram. Note that by definition we have $\beta \cdot \mu=0$ and $\beta \cdot \lambda=1$ since if we forget the $z$ basepoint the $\beta$ curve is isotopic to the meridian. We orient the meridian as shown in figure \ref{P_3doublypointeddiagram} and the $\beta$ curve inherits an induced orientation from the meridian.

\begin{figure}
\begin{center}
\begin{tikzpicture}
\node at (5,7.8) {$\lambda=\alpha_1^a$};
\node at (1,4) {$\mu=\alpha_2^a$};
\node[anchor=south west,inner sep=0] at (0,0)    {\includegraphics[scale=.22]{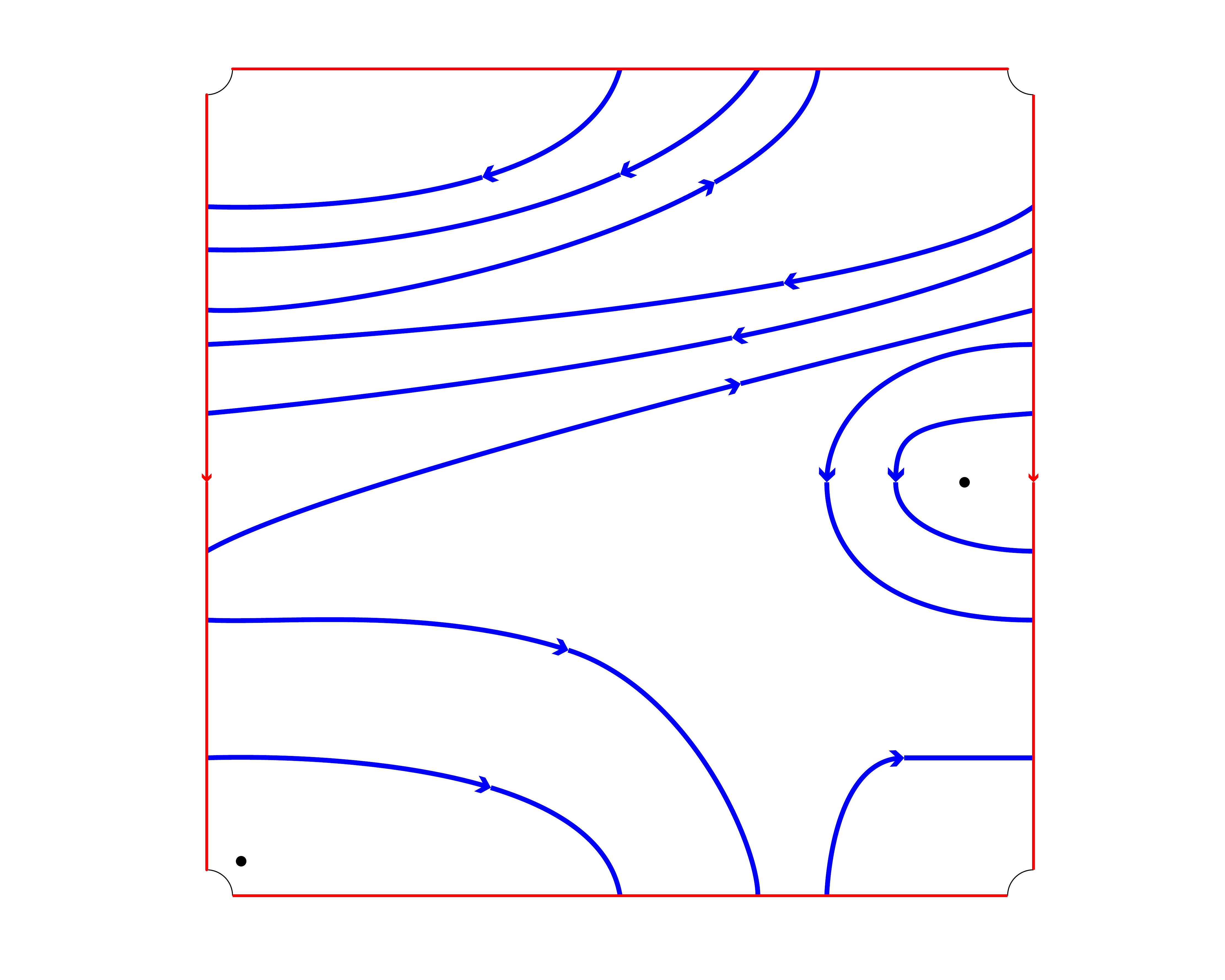}};
\node[font=\footnotesize] at (5,.5) {$x_0$};
\node[font=\footnotesize] at (6.1,.5) {$x_1$};
\node[font=\footnotesize] at (6.8,.5) {$x_2$};
\node[font=\footnotesize] at (2,1.2) {$w$};
\node[font=\footnotesize] at (8.6,6.3) {$y_0$};
\node[font=\footnotesize] at (8.6,5.9) {$y_1$};
\node[font=\footnotesize] at (8.6,5.4) {$y_2$};
\node[font=\footnotesize] at (8.6,5.1) {$y_3$};
\node[font=\footnotesize] at (8.6,4.6) {$y_4$};
\node[font=\footnotesize] at (8.6,3.45) {$y_5$};
\node[font=\footnotesize] at (8.6,2.95) {$y_6$};
\node[font=\footnotesize] at (8.6,1.8) {$y_7$};
\node[font=\footnotesize] at (8.1,4) {$z$};
\end{tikzpicture}
\caption{The genus $1$ doubly pointed Heegaard diagram for the pattern $P_{3,1}$}\label{P_3doublypointeddiagram}
\end{center}
\end{figure}
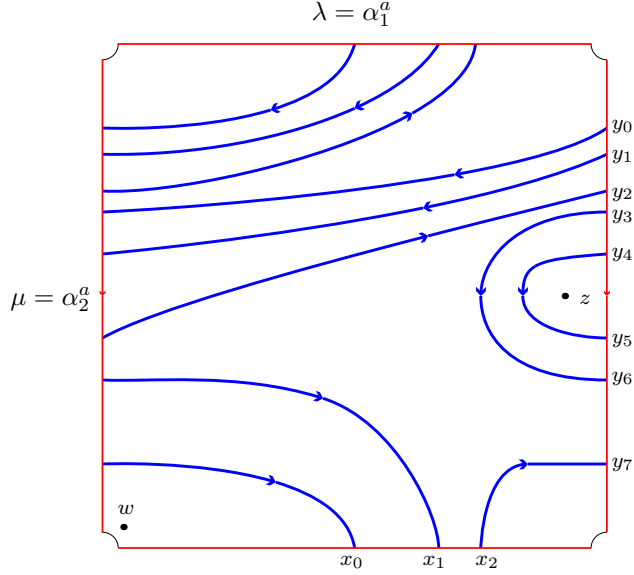

Now we describe how to obain $\CFAhat(S^1\times D^2,P)$ from a given genus $1$ doubly pointed bordered Heegaard diagram. As an $\F$ vector space $\CFAhat(S^1\times D^2,P)$ is generated by elements of the set 
$$\mathcal{G}=\{x| x \in\beta \cap \alpha^a\}.$$ For each $x \in\mathcal{G}$, we have the following right action of the idempotent subalgebra $\mathcal{I}$: $x\cdot \iota_0=x$ if $x \in \alpha_1^a \cap \beta$ and $x \cdot \iota_0=0$ otherwise. Similarly, $x \cdot \iota_1=x$ if $x \in \alpha_2^a \cap \beta$ and $x \cdot \iota_1=0$ otherwise. 

Now, regard the surface-with-boundary $\Sigma$ as $T^2 \setminus D^2$. Let $\R^2 \to T^2$ denote the universal cover of the torus, and set $\tilde{\Sigma}$ to be the covering space obtained from $\R^2$ by removing the lifts of $D^2$. Using this covering space, we define the maps
$$m_{n+1}: M \otimes \mathcal{A}^{\otimes n} \to M$$ for $n \geq 0$ as follows. 

$$m_{n+1}(x,\rho_{I_1},\cdots,\rho_{I_n})=\sum_{y \in \mathcal{G}} \# \mathcal{M}(x,y)y$$ where $\#\mathcal{M}(x,y)$ is the mod $2$ count of index $1$ embedded disks in $\tilde{\Sigma}$ such that, when we traverse the boundary of the disk we start from a lift of $x$ and walk along an arc of some lift of $\alpha^a$ then along the arc $\rho_{I_1}$ on some lift of $\partial D^2$, \dots, then walk along some the arc $\rho_{I_n}$ and then along some lift of $\alpha^a$ to $y$ and finally along a lift of $\beta$ from $y$ to $x$. 

For example, consider the doubly pointed genus $1$ Heegaard diagram shown in figure \ref{P_3doublypointeddiagram}. The generators of $\CFAhat(P)$ in idempotent $\iota_0$ (intersection of $\beta$ with $\alpha^a_1$) are labelled $x_0,x_1,x_2$ from left to right and the generators in idempotent $\iota_1$ (intersections of $\beta$ with $\alpha_2^a$) are labelled $y_0, \cdots, y_7$ from top to bottom. We draw the lift to the cover $\tilde{\Sigma}$ in figure \ref{liftofP_{3,1}pattern} and indicate a few of the type $A$ operations given by the disks shown. The gray disk gives a $m_3(x_0,\rho_{12},\rho_1)=y_3$, the green disk gives $m_2(x_1,\rho_1)=y_1$ and the pink disk gives $m_3(y_1,\rho_2,\rho_1)=y_4$. The full type $A$ module $\CFAhat(S^1\times D^2,P_{3,1})$ is shown in figure \ref{CFAhat}. In that figure, an arrow labelled $\rho_{I_1},\rho_{I_2},\dots,\rho_{I_n}$ from $x$ to $y$ means there is an $A_{\infty}$ operation $m_{n+1}(x,\rho_{I_1},\dots,\rho_{I_n})=y$.

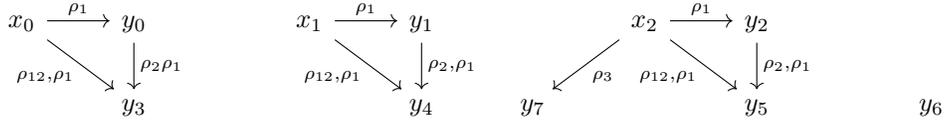
\begin{figure}[!tbp]
\begin{center}
\begin{tikzcd}
x_0 \arrow{dr}[swap]{\rho_{12},\rho_1} \arrow{r}{\rho_1} & y_0 \arrow{d}{\rho_{2}\rho_1} && x_1 \arrow{dr}[swap]{\rho_{12},\rho_1}\arrow{r}{\rho_1} &y_1 \arrow{d}{\rho_2,\rho_1} && x_2 \arrow{dr}[swap]{\rho_{12},\rho_1} \arrow{dl}{\rho_3}\arrow{r}{\rho_1} &y_2 \arrow{d}{\rho_2,\rho_1}\\

& y_3 &&& y_4 &y_7&  & y_5&& y_6 \\
\end{tikzcd}
\caption{$\CFAhat(\mathcal{H})$ where $\mathcal{H}$ is the doubly pointed bordered Heegaard diagram shown in figure \ref{P_3doublypointeddiagram}}\label{CFAhat}
\end{center}
\end{figure}

\begin{figure}
\begin{center}
\begin{tikzpicture}\hspace{-.9in}
\node[anchor=south west,inner sep=0] at (0,0)    {\includegraphics[scale=.4]{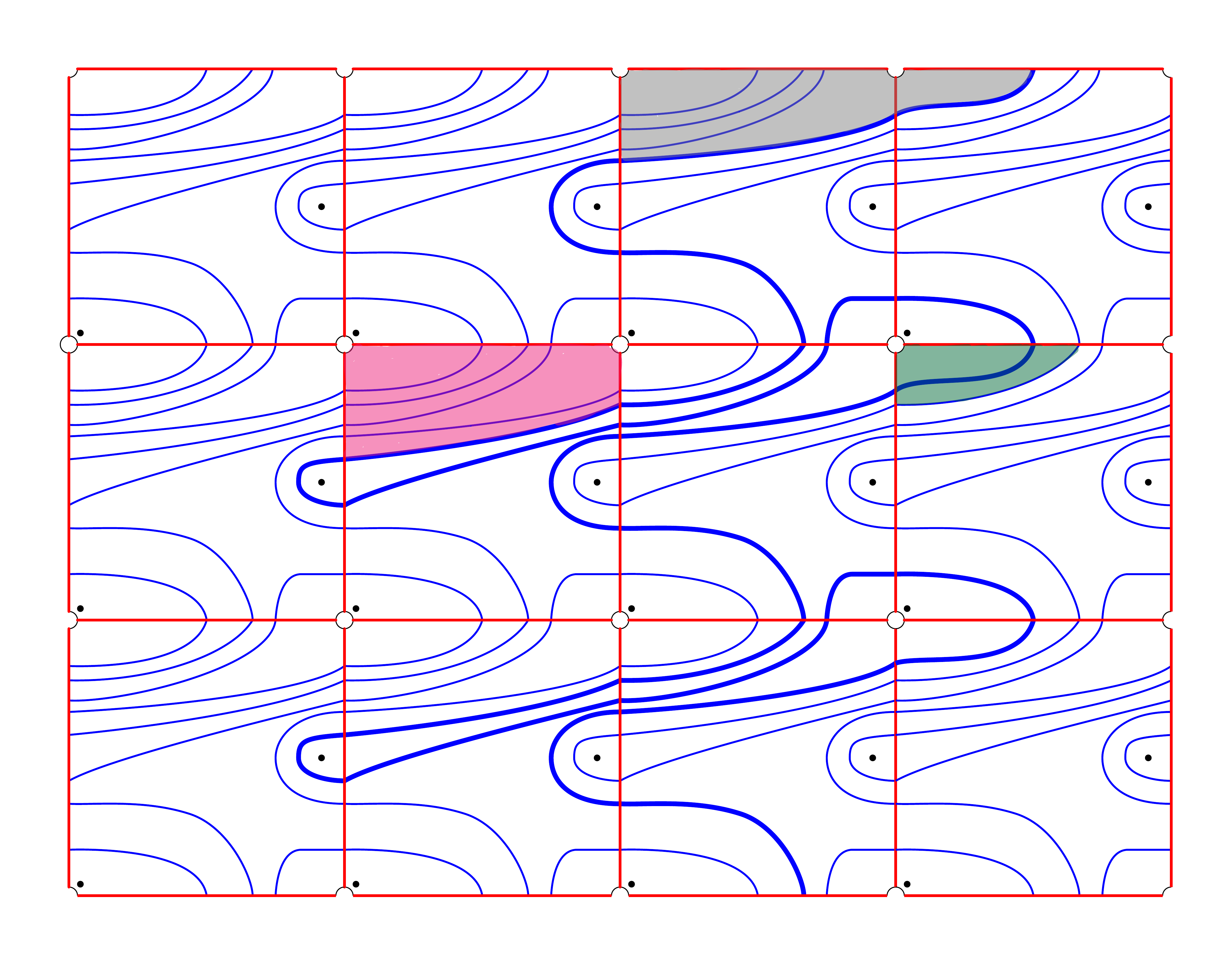}};
\node at (15.1,13.6) {$x_0$};
\node at (9,11.9) {$y_3$};
\node at (15.75,9.5) {$x_1$};
\node at (13,8.25) {$y_1$};
\node at (9.45,8.3)  {$y_1$};
\node at (4.85,7.8) {$y_4$};
\end{tikzpicture}
\caption{Lift of the pattern $P_{3,1}$ to the cover $\tilde{\Sigma}$ a single connected lift of $\beta$ is shown in bold}\label{liftofP_{3,1}pattern}
\end{center}
\end{figure}
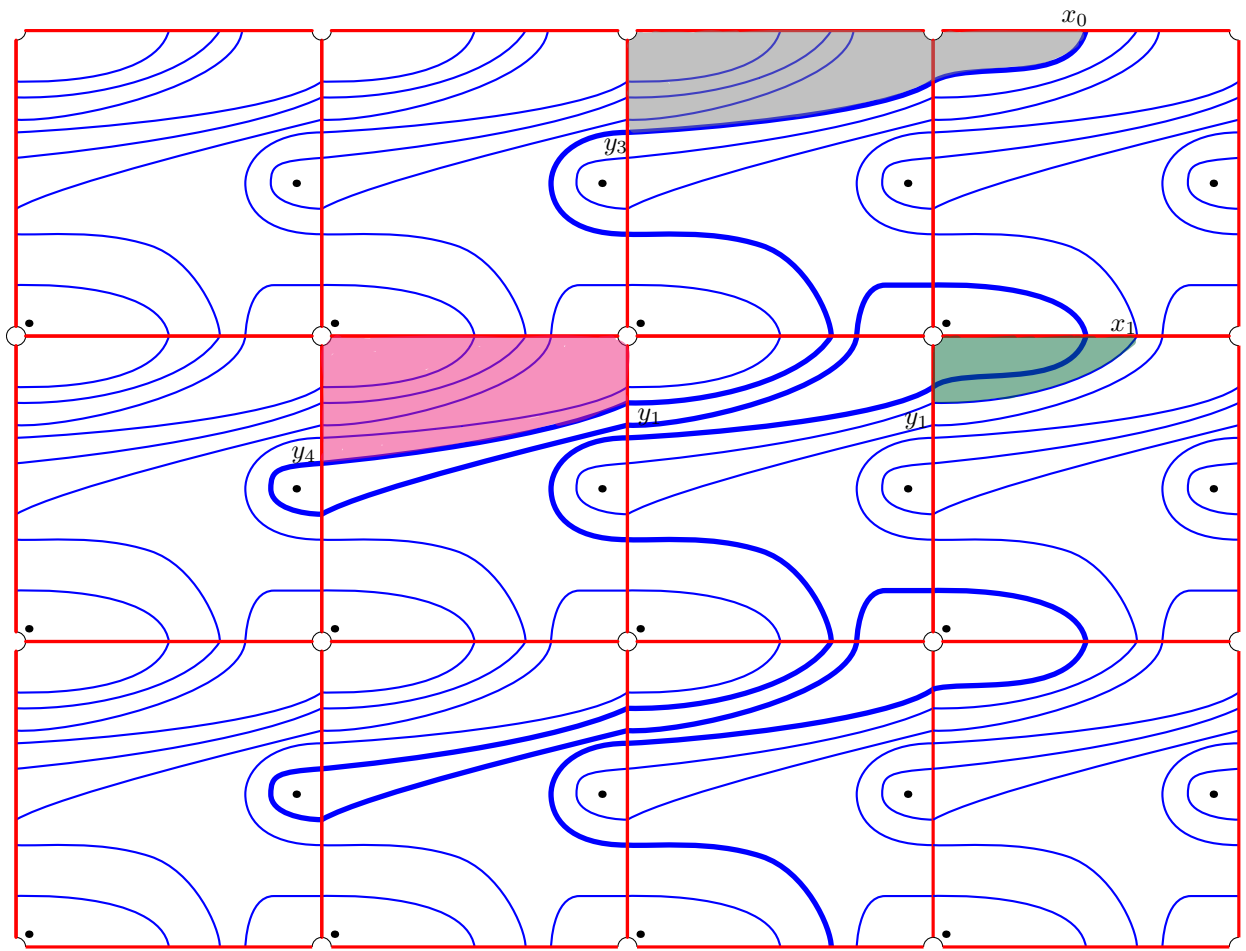

\section{The pairing theorem for $(1,1)$ patterns}\label{(1,1)pairingtheorem}

\begin{figure}[!tbp]
\begin{center}
\begin{tikzpicture}
\node[anchor=south west, inner sep=0] at (0,0) {\includegraphics[scale=.2]{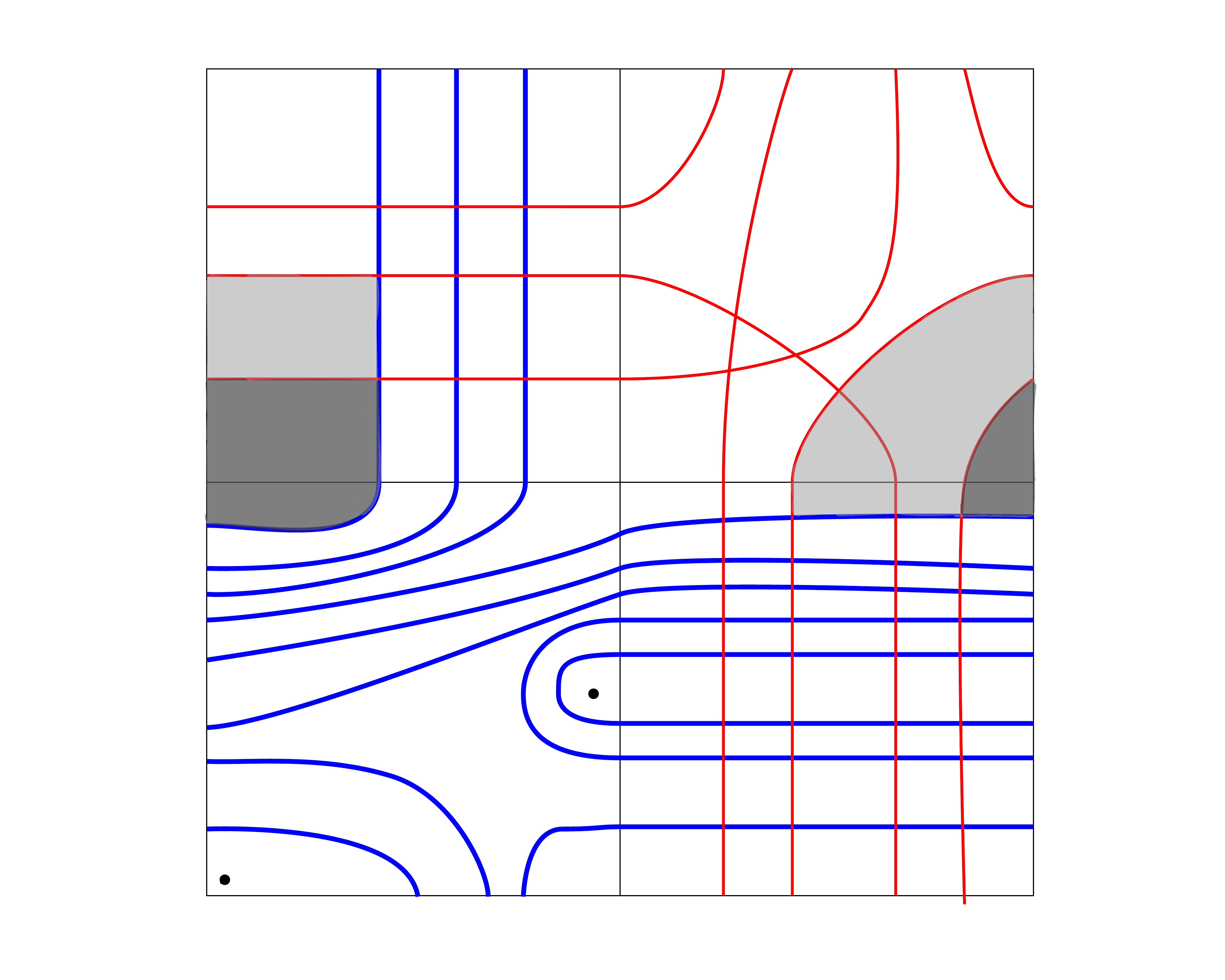}};
\node[font=\footnotesize]at (7.1,3.4) {$\tikzcircle{2pt}$};
\node[font=\tiny]at (7.25,3.2) {$y_0\kappa$};
\node at (5.85,3.35) {$\tikzcircle{2pt}$};
\node[font=\tiny] at (6.2,3.25) {$y_0\mu_1$};
\node[font=\tiny] at (2.8,4.2) {$x_0\xi_1$};
\node[font=\tiny] at (2.8,4.45) {$\tikzcircle{2pt}$};
\node[font=\tiny] at (2.8,5.2) {$\tikzcircle{2pt}$};
\node[font=\tiny] at (2.8,5) {$x_0\xi_0$};
\end{tikzpicture}
\caption{pairing diagram showing the trefoil pattern $P_{3,1}$ paired with $0$ framed trefoil companion}\label{P3pairingwith0framedtrefoil}
\end{center}
\end{figure}

The main result in \cite{Chen} is a reinterpretation of the pairing theorem from \cite{LOT}*{Theorem 11.19} in terms of immersed curves when $\CFAhat(S^1\times D^2,P)$ comes from a $(1,1)$ pattern $P$. In this section we recall this theorem. 

Let $\beta(P)$ denote the $\beta$ curve in the data of a genus one doubly pointed Heegaard diagram and let $\alpha(K)$ denote the immersed curve for $S^3\setminus \nu(K)$ as described in section \ref{immersedD}. Chen's theorem says that to compute $\HFKhat(S^3,P(K))$ we can compute the intersection Floer homology of $\alpha(K)$ and $\beta(P)$, denoted $\CFKhat(\alpha,\beta)$, in the torus as follows. Let $T^2=[0,1]^2/\sim$ and divide the square into four quadrants. Include the immersed curve $\alpha(K)$ into the first quadrant $[\frac{1}{2},1]\times[\frac{1}{2},1]$) and include $(\beta(P),w,z)$ into the third quadrant. Then extend both curves horizontally and vertically, so that $\alpha$ and $\beta$ intersect in the second and fourth quadrants only. In this set up intersections in the second quadrant correspond to generators of $\CFAhat(S^3,P)\boxtimes \CFDhat(S^3\setminus \nu(K))$ that come from pairing generators in idempotent $\iota_0$ and intersection points in the fourth quadrant correspond to generators of $\CFAhat(S^3,P)\boxtimes \CFDhat(S^3\setminus \nu(K))$ that come from pairing generators in the $\iota_1$ idempotent. The main work in \cite{Chen} is constructing from a differential in the Lagrangian Floer chain complex, $\CFKhat(\alpha,\beta)$, a type A operation in $\CFAhat(S^1\times D^2,P)$ and a corresponding type D operation in $\CFDhat(S^3\setminus \nu(K))$ so that these pair in the box tensor product to produce the given diferential.

The data of the torus divided into quadrants, with the curves $\alpha(K)$ and $\beta(P)$ places as described, or this same picture lifted to the universal cover, will be referred to as a \textit{pairing diagram} for the knot Floer homology of the satellite $P(K)$. For an example of a pairing diagram for the knot Floer homoloy of the satellite knot $P_{3,1}(T_{2,3})$ see figure \ref{P3pairingwith0framedtrefoil}. From the picture we can see that $\CFKhat(S^3,P_{3,1}(T_{2,3}))$ has $41$ generators. In that figure, we also indicate two differentials, in light and dark grey, that contribute to $\partial^{\boxtimes}$. The dark grey disk gives a differential in $\CFKhat(S^3,P_{3,1}(T_{2,3}))$ connecting $x_0\boxtimes \xi_1$ to $y_0 \boxtimes \kappa$. This arises from pairing the type A operation $m_2(x_0,\rho_1)=y_0$ and the Type D operation $\delta(\xi_1)=\rho_1\otimes\kappa$. The light grey disk represents a differential from $x_0 \boxtimes\xi_0$ to $y_0 \boxtimes \mu_1$ given by pairing the type A operation $m_2(x_0,\rho_1)=y_0$ and the type D operation $\delta(\xi_0)=\rho_1 \boxtimes \mu_1$.

For convenience we will usually draw pictures of single lifts of $\alpha(K)$ and $\beta(P)$ to the universal cover $\pi: \R^2 \to T^2$ of the torus. Here we choose a single lift of $\beta(P)$, call it $\tbeta$, and a lift of $\alpha(K)$, call it $\talpha$, so that $\talpha$ is in pegboard position with respect to a peg at the midpoint of the arc $\delta_{w,z}$ of large enough radius to contain both basepoints $w$ and $z$. We also require that $\talpha$ and $\tbeta$ intersect transversely and there are no pairs of intersections that are connected by a Whitney disk that does not cross any basepoint. This is allowed, since intersection Floer homology is an isotopy invariant (a topic we come back to in the next section). These conditions ensure that $\CFKhat(\talpha,\tbeta)\simeq\HFKhat(S^3,P(K))$, where $\CFKhat(\talpha,\tbeta)$ denotes the intersection Floer homology of the two curves in $\R^2\setminus \pi^{-1}(\{w,z\})$. See figure \ref{P_3trefoilpairingdiagram} for an example computing $\HFKhat(S^3,P_{3,1}(T_{2,3}))$ from a lifted pairing diagram. This figure shows that $\rk(\HFKhat(S^3,P_{3,1}(T_{2,3})))=21$.


The last bit of information we want to extract from the pairing diagram is the Alexander grading on $\HFKhat(S^3,P(K))$. This is achieved by the following lemma.

\begin{lemma}[\cite{Chen} Lemma 4.1]\label{Alex}

Let $x$ and $y$ be two intersection points between $\alpha$ and $\beta$. Let $\ell$ be an arc on $\beta$ from $x$ to $y$, and let $\delta_{w,z}$ be a straight arc connecting $w$ to $z$. Then $$A(y)-A(x)=\ell\cdot\delta_{w,z}$$

\end{lemma}

For example, consider the intersection points labelled $x$ and $y$ in figure \ref{P_3trefoilpairingdiagram}. These intersection points are connected by an arb of the $\beta$ curve that is shown in bold in the figure. When we traverse this arc, from $x$ to $y$ along the orientation of $\beta$, we cross five $\delta_{w,z}$ arcs positively. Then Lemma \ref{Alex} implies $A(y)-A(x)=5$. 



\subsection{Computing $\tau(P(K))$ from a pairing diagram}\label{tauofpairing}

In this section, we recall from \cite{Chen} the precedure for computing $\tau$ from the pairing diagram for $\CFKhat(P(K))$ when $P$ is a $(1,1)$ pattern. Recall that the Alexander filtration on $\CFKhat(K)$ produces a spectral sequence converging to $\HFhat(S^3)$. The $\tau$ invariant is the minimal Alexander grading of the cycle that survives to the $E^{\infty}$ page. In what follows we give a way of computing this spectral sequence in the pairing diagram for $\CFKhat(S^3,P(K))$ for $(1,1)$ patterns $P$. First, we recall the following well known lemma that gives us a way of thinking about passing from one page of the spectral sequence to the next as cancelling differentials that decrease filtration by the minimal amount, see \cite{BLH,Zhang}.

\begin{lemma}[\cite{BLH}*{Lemma 2.4}]\label{cancel}
Suppose $(C,d)$ is a chain complex over $\F_2$ freely generated by elements $\{x_i\}$. Let $d(x_i,x_j)$ be the coefficient of $x_j$ in $d(x_i)$ and suppose $d(x_k,x_l)=1$. Then the complex $(C',d')$ wih generators $\{x_i| i \neq k,l\}$ and differential 

$$d'(x_i)=d(x_i)+d(x_i,x_l)d(x_k)$$

is chain homotopy equivalent to $(C,d)$
\end{lemma}

Now, suppose $C$ is a filtered chain complex. The above lemma tells us how to compute the spectral sequence associated to the filtration in stages. The $E_0$ term of the spectral sequence is the associated graded $C=\bigoplus C_i$. Then, we pass from the $E_0$ to the $E_1$ page by cancelling the components of $d$ that do not shift the grading, and arrive at a chain complex $(E_1,d_1)$, where the $d_1$ differential is defined as in lemma \ref{cancel}. Continuing in this way, we pass from the $E_1$ page to the $E_2$ page by cancelling the components of the differential $d_1$ that shift grading by one, etc. In this way, the spectral sequence collapses when we have reached a chain complex filtered chain homotopy equivalent to the original one but whose differential is zero. For more details, see the discussion after Remark 2.5 in \cite{BLH}.

In the spectral sequence induced by the Alexander filtration on $\CFKhat(\talpha,\tbeta)$, the previous discussion shows that passing from one page to the next in this spectral sequence amounts to cancelling differentials that connect elements of minimal Alexander filtration difference. We now give an way to see that cancellation geometrically in the complex $\CFKhat(\talpha,\tbeta)$. In the pairing diagram, differentials are given by Whitney disks that connect two intersection points and cross the $z$ basepoint, but not the $w$ basepoint and the filtration difference is the number of $z$ basepoints enclosed. To cancel two generators connected by such a Whitney disk, we perform an isotopy of the $\beta$ curve over the disk to a new curve $\beta'$ thus cancelling those two intersection points in the diagram, together with possible more if the Whitney disk contains any arcs of the $\alpha$ curve in its interior. In any case, all the intersection points cancelled by isotoping away this Whitney disk will all have the same filtration difference, so it doesn't matter if we cancel pairs of generators of minimal filtration difference one at a time or in bulk. Once this isotopy is performed, we arrive at a new complex, with fewer generators. To remember the filtration difference after the cancellation, following Chen we place small arrows on the $\beta'$ curve, called $A$-bouys, that remember that an isotopy of a Whitney disk crossing some number of $z$ basepoints was performed. Then, when we compute filtration differences of the remaining intersection points in the $\alpha$ and $\beta'$ complex, we count both intersections of the $\beta'$ curve with the $\delta_{w,z}$ arcs and the $A$-bouys.

It remains to observe that when we cancel two intersection points by isotoping the curve $\beta$ to $\beta'$, the differential $d'$ on the Lagrangian Floer chain complex $\CFKhat(\talpha,\tbeta')$, which is given by counting holomorphic disks with boundary conditions on $\talpha$ and $\tbeta'$, is given by the formula $d'(a)=d(a)+d(a,y)d(x)$, where $d(a,y)$ is, as above, the coefficient of $y$ in $d(a)$. To see this, recall that for a generator $x$ of the Lagrangian Floer chain complex, the differential is given by

$$d(x)=\sum_y n(x,y)y$$ where $n(x,y)$ counts Maslov index $1$ holomorophic disks connecting $x$ to $y$ in the $\alpha, \beta$ complex. Now, suppose that we isotope the curve $\beta$ to a new curve $\beta'$ where $\beta'$ results from isotoping $\beta$ over a Whitney disk that crosses the $z$ basepoint and cancels the intersection points $x$ and $y$ of minimal filtration difference. Then by \cite{combfloer}*{equation 59}, the new holomorphic disk count in the $\alpha$ and $\beta'$ complex is given by

$$n'(a,b)=n(a,b)+n(a,y)n(x,b)$$

Where $a,b \in \alpha \cap \beta'$. This implies that $d'(a)=d(a)+d(a,y)d(x)$ for $a \in \CFKhat(\talpha,\tbeta')$. Indeed, we have
$$d'(a)=\sum_b n'(a,b)b=\sum_b n(a,b)b+n(a,y)\sum_b n(x,b)b=d(a)+d(a,y)d(x).$$

\begin{figure}[!tbp]
  \centering
  \begin{minipage}[b]{0.3\textwidth}
  \begin{tikzpicture}
\node[anchor=south west, inner sep=0] at (0,0) { \includegraphics[width=1.2\textwidth]{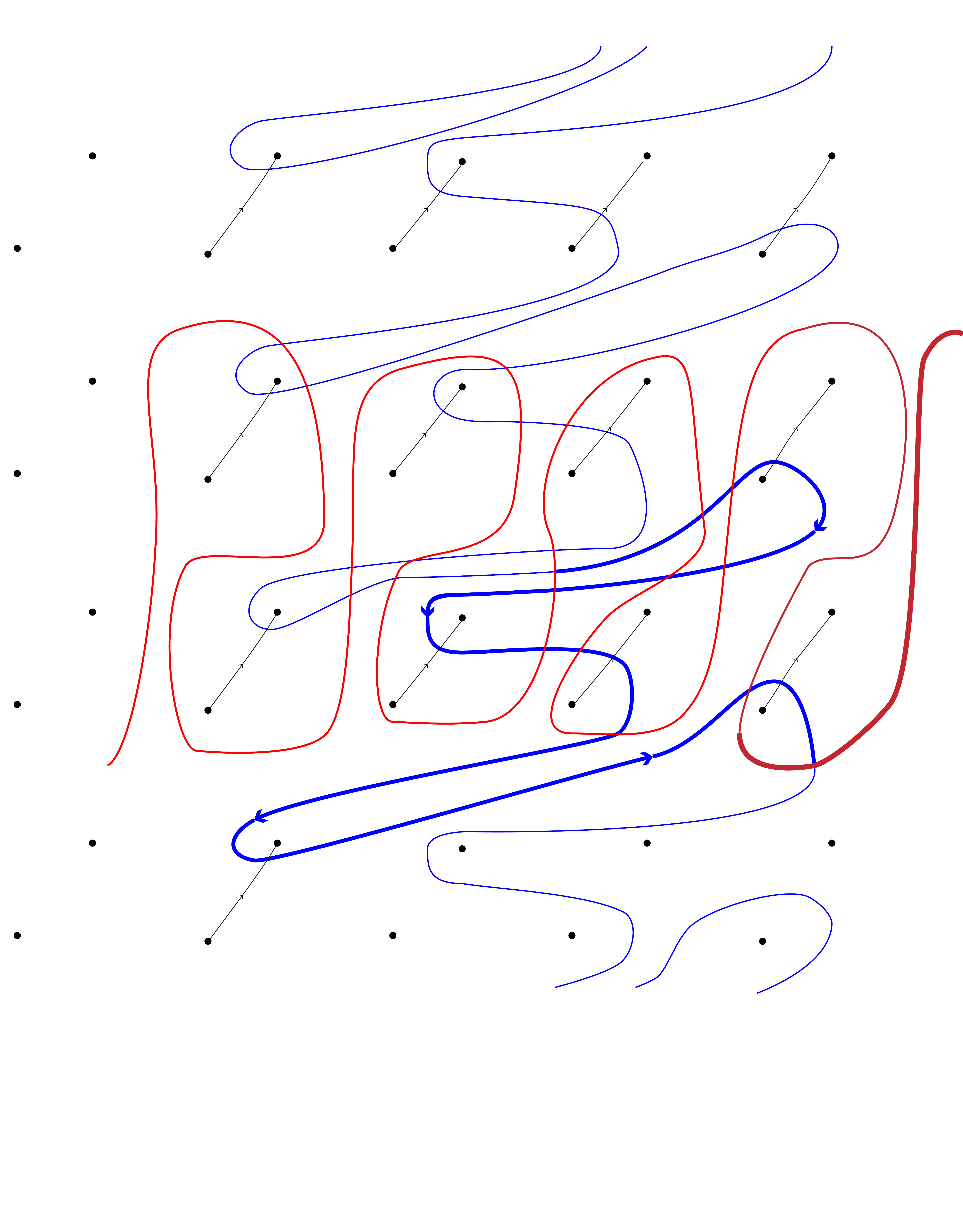}};
\node[font=\tiny] at (2.82,3.5) {$x$};
\node at (2.9,3.45) {$\tikzcirclee{1pt}$};
\node at (4.29,2.45) {$\tikzcirclee{1pt}$};
\node[font=\tiny] at (4.35,2.35) {$y$};

\node[font=\tiny] at (1,1.4) {$w$};
\node[font=\tiny] at (1,2.75) {$w$};
\node[font=\tiny] at (1,3.9) {$w$};
\node[font=\tiny] at (1,5.2) {$w$};

\node[font=\tiny] at (1.5,5.3) {$\delta_{w,z}$};
\node[font=\tiny] at (2.5,5.3) {$\delta_{w,z}$};
\node[font=\tiny] at (3.5,5.3) {$\delta_{w,z}$};
\node[font=\tiny] at (4.5,5.3) {$\delta_{w,z}$};

\node[font=\tiny] at (1.5,4) {$\delta_{w,z}$};
\node[font=\tiny] at (1.5,2.8) {$\delta_{w,z}$};
\node[font=\tiny] at (1.5,1.5) {$\delta_{w,z}$};

\node[font=\tiny] at (1.4,2.1) {$z$};
\node[font=\tiny] at (1.35,3.25) {$z$};
\node[font=\tiny] at (1.35,4.5) {$z$};
\node[font=\tiny] at (1.35,5.6) {$z$};

\node[font=\tiny] at (2.4,2) {$z$};
\node[font=\tiny] at (2.35,3.25) {$z$};
\node[font=\tiny] at (2.35,4.45) {$z$};
\node[font=\tiny] at (2.35,5.6) {$z$};

\node[font=\tiny] at (3.5,2) {$z$};
\node[font=\tiny] at (3.45,3.25) {$z$};
\node[font=\tiny] at (3.45,4.45) {$z$};
\node[font=\tiny] at (3.45,5.6) {$z$};

\node[font=\tiny] at (4.5,2) {$z$};
\node[font=\tiny] at (4.45,3.25) {$z$};
\node[font=\tiny] at (4.45,4.45) {$z$};
\node[font=\tiny] at (4.45,5.6) {$z$};

\node[font=\tiny] at (2,1.4) {$w$};
\node[font=\tiny] at (2.2,2.75) {$w$};
\node[font=\tiny] at (2,3.9) {$w$};
\node[font=\tiny] at (2,5.2) {$w$};

\node[font=\tiny] at (3,1.4) {$w$};
\node[font=\tiny] at (3.2,2.75) {$w$};
\node[font=\tiny] at (3,3.9) {$w$};
\node[font=\tiny] at (3,5) {$w$};

\node[font=\tiny] at (4,1.4) {$w$};
\node[font=\tiny] at (4.05,2.5) {$w$};
\node[font=\tiny] at (4,3.8) {$w$};
\node[font=\tiny] at (4,5) {$w$};

\end{tikzpicture}
    \caption{Pairing diagram for $\HFKhat(S^3,P_{3,1}(T_{2,3}))$. Intersection points labelled $x$ and $y$ satisfy $A(y)-A(x)=5$ and $A(x)=0$.}\label{P_3trefoilpairingdiagram}
  \end{minipage}
  \hspace{2in}
  \begin{minipage}[b]{0.3\textwidth}
  \begin{tikzpicture}
\node[anchor=south west,inner sep=0] at (0,0)    {\includegraphics[width=1.2\textwidth]{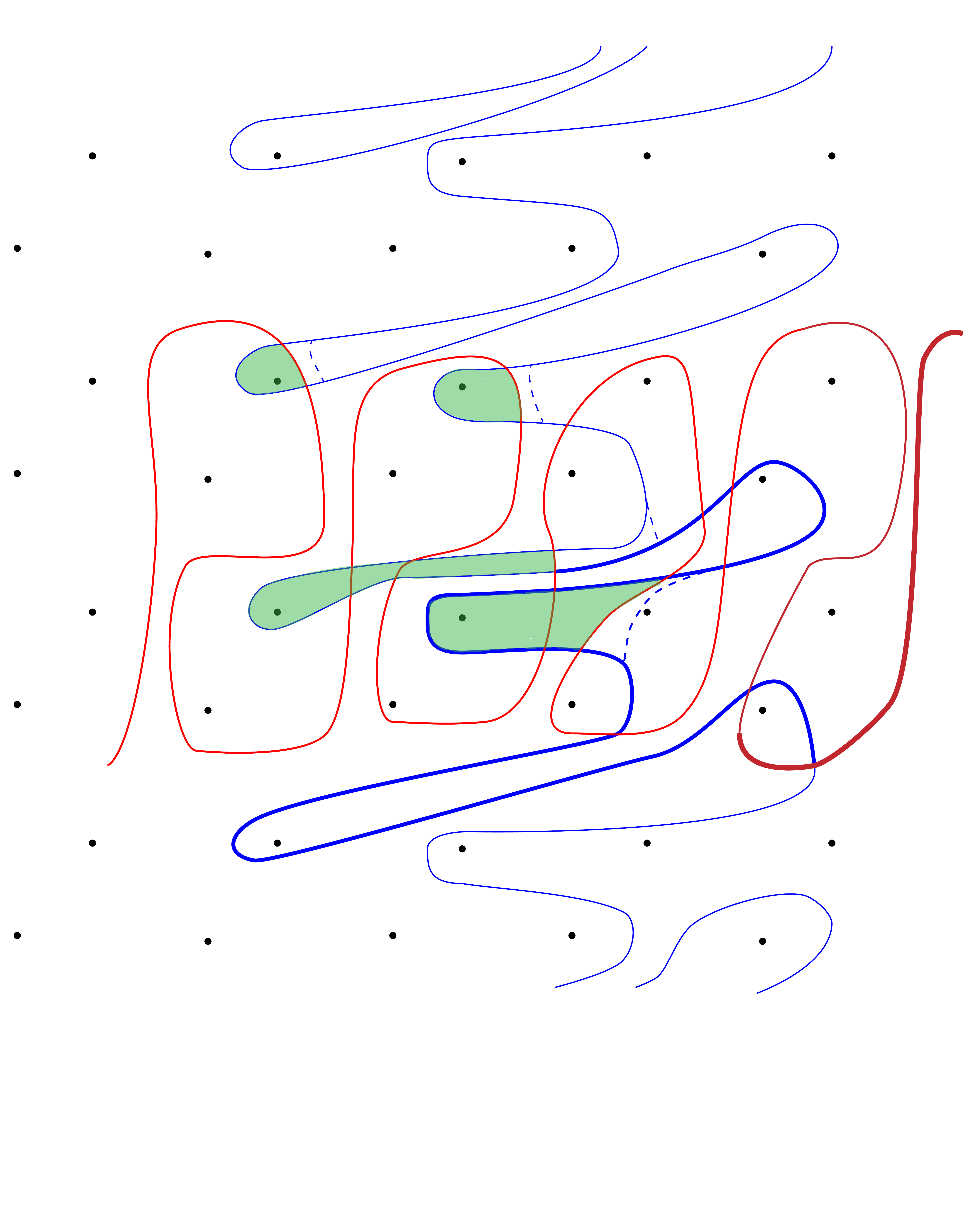}};
\end{tikzpicture}
    \caption{The disks shown represent all the differentials that lower filtration degree by one. Cancelling the disks by an isotopy, we end up with figure \ref{cancellengthtwo}}\label{calcellengthone}
  \end{minipage}
  \end{figure}
  
\begin{figure}[!tbp]
  \centering
  \begin{minipage}[b]{0.3\textwidth}
  \begin{tikzpicture}
\node[anchor=south west, inner sep=0] at (0,0)  {\includegraphics[width=1.2\textwidth]{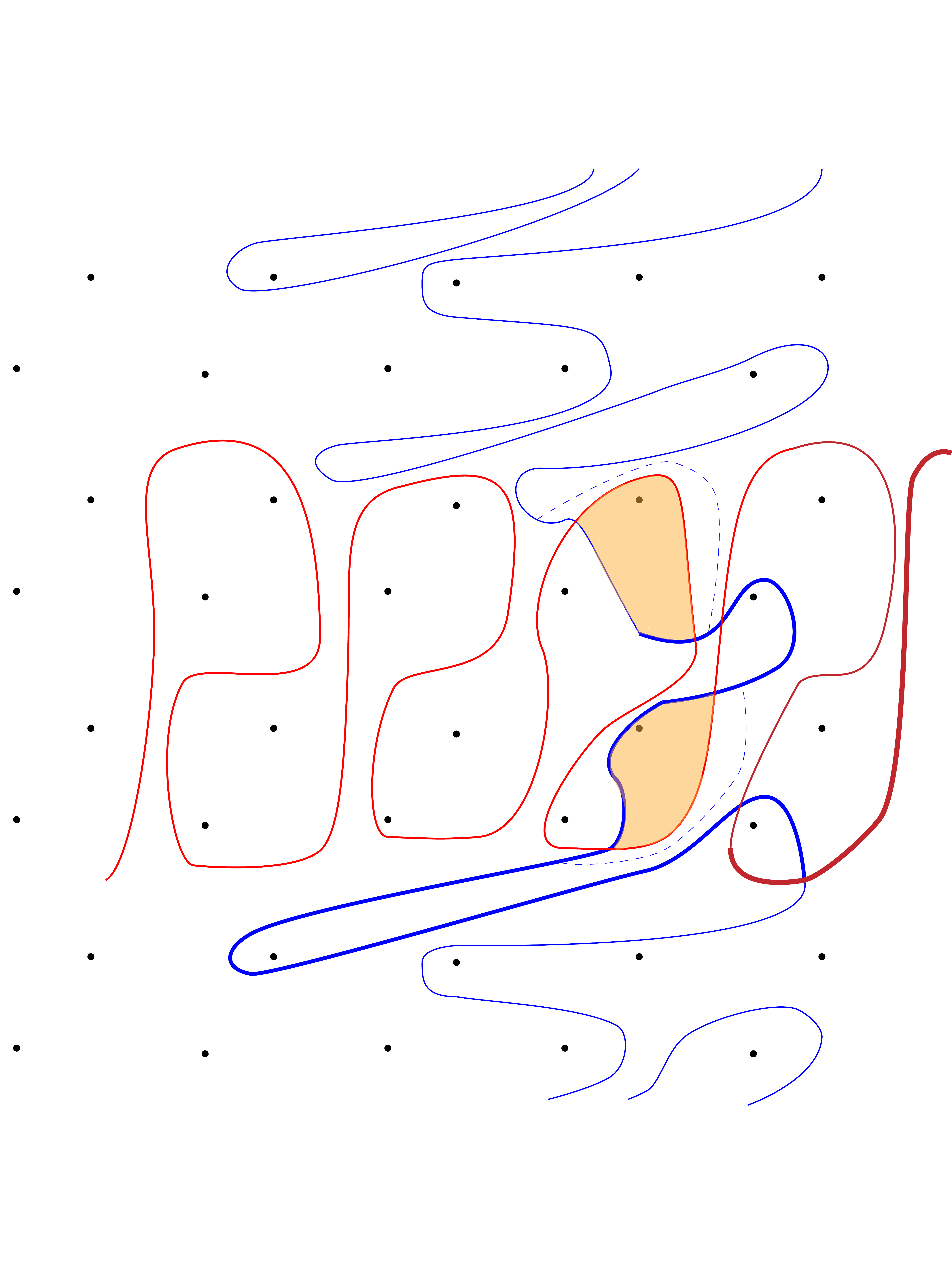}};
\end{tikzpicture}
    \caption{The result of cancelling the disks in figure \ref{calcellengthone}. There are two disks that connect generators of $\CFKhat(\talpha,\tbeta')$ of minimal filtration difference. Cancelling these disks we arrive at figure \ref{deepestcancel}}\label{cancellengthtwo}
  \end{minipage}
  \hspace{2in}
  \begin{minipage}[b]{0.3\textwidth}
  \begin{tikzpicture}
\node[anchor=south west,inner sep=0] at (0,0)    {\includegraphics[width=1.2\textwidth]{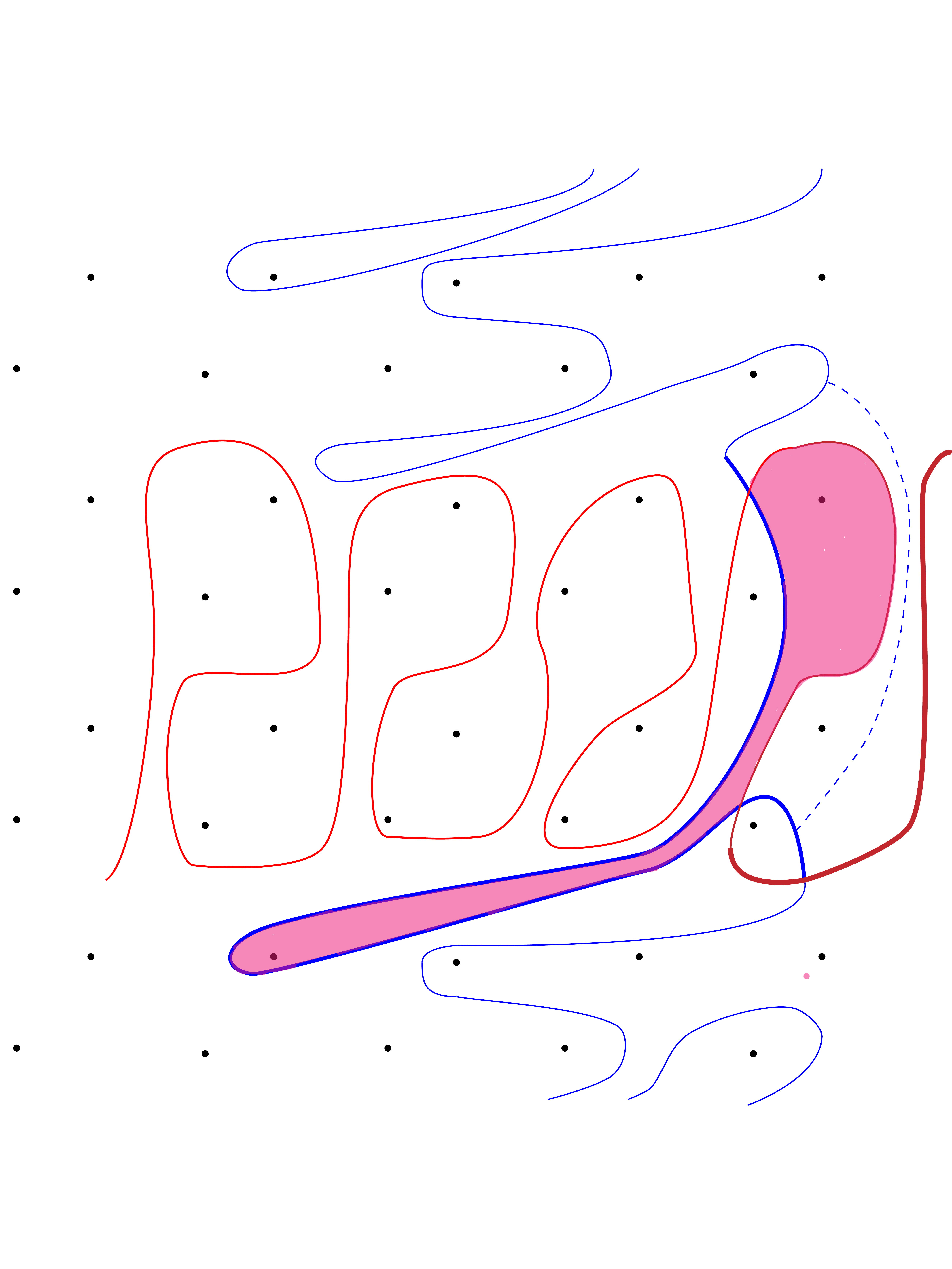}};
\end{tikzpicture}
    \caption{The result of isotoping $\beta'$ in figure \ref{cancellengthtwo}, we arrive at a complex with three generators and one differential connecting two generators of minimal filtration difference}\label{deepestcancel}
  \end{minipage}
\end{figure}

This gives a diagramatic way to run the Alexander filtration spectral sequence in a pairing diagram. For example, consider figures \ref{P_3trefoilpairingdiagram}-\ref{deepestcancel}. In that sequence of figures we first see the pairing diagram for $\HFKhat(S^3,P_{3,1}(T_{2,3}))$ in figure \ref{P_3trefoilpairingdiagram}. In figure \ref{calcellengthone}, we have indicated all of the Whitney disks that connect two intersection points of filtration difference one. When we cancel these disks by isotoping the $\beta$ curve over these disks, we arrive at figure \ref{cancellengthtwo}. In that figure, we have indicated the disks that connect intersection points of minimal filtration difference. Cancelling these, we arrive at figure \ref{deepestcancel}, where we see three intersection points, two of which are connected by a Whitney disk, shown in the figure in purple. If we cancel these two generators we arrive at a pairing diagram with one intersection point. The Alexander grading of this intersection point is then $\tau(P_{3,1}(T_{2,3}))$ by the discussion above.

\begin{figure}
  \begin{tikzpicture}
\node[anchor=south west,inner sep=0] at (0,0)    {\includegraphics[width=\textwidth]{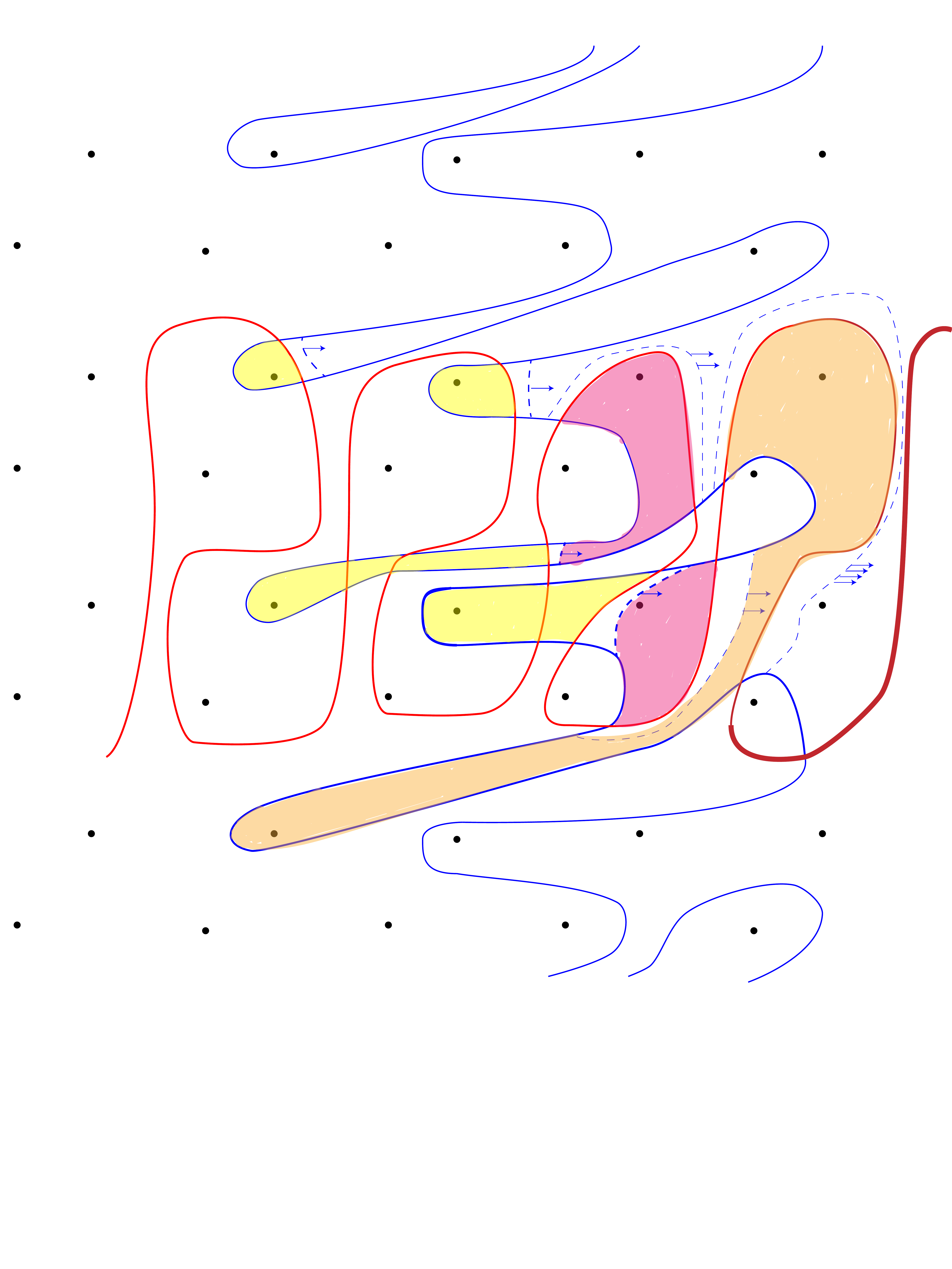}};
\node[font=\footnotesize] at (11.75,7.5) {$\tikzcirclee{3pt}$};
\node[font=\small] at (12,7.3) {$\textcolor{green}{a}$};
\node at (10.9,8.6) {$\tikzcirclee{3pt}$};
\node[font=\small] at (11.1,8.5) {$\textcolor{green}{b}$};
\node at (10.7,11.6) {$\tikzcirclee{3pt}$};
\node[font=\small] at (10.9,11.6) {$\textcolor{green}{c}$};
\node at (8.05,10.35) {$\tikzcircle{3pt}$};
\node[font=\small] at (7.9,10.2) {$x$};
\end{tikzpicture}\caption{Cancelling all intersection points with filtration difference one (disk in yellow) and intersection points with filtration difference two (disks in pink) There are three intersection points remaining.}\label{P3pairingRHTwithbouys}
\end{figure}
A convenient way to package the entire spectral sequence is shown in figure \ref{P3pairingRHTwithbouys}. Here we see all of the disks we cancelled in the spectral sequence, and the A bouys that keep track of the Alexander filtration from the original complex in all of the subsequent pages. Note that if we draw it like this, we have to cancel all intersection points with filtration difference one before cancelling any with filtration difference two, etc. We can find the absolute Alexander grading of the generator labelled $a$ (the intersection point we found to survive the $z$-basepoint spectral sequence) as follows. By the symmetry under the elliptic involution, it is easy to see that $A(x)=0$. Then using Lemma \ref{Alex} we have $A(a)-A(x)=A(a)=5$, so $\tau(P_{3,1}(T_{2,3}))=5$


\section{Trefoil patterns}\label{trefoilpatterns}

In this section we will compute $\tau$ of satellite knots with arbitrary companion knot $K$ and pattern $P$ from a family of trefoil patterns that we will now describe. 

\subsection{Introducing the patterns}
The $(1,1)$-patterns studied in this paper are constructed by isotoping the $\beta$ curve on the doubly pointed bordered Heegaard diagram for the unknotted $(p,1)$-cable pattern so that we introduce only two extra intersection points between $\beta$ and $\lambda$. To describe the isotopy, consider first the case $p=3$. The unknotted $(3,1)$-cable pattern is shown in figure \ref{(3,1)cable}. Isotope the $\beta$ curve by taking the bottom-most horizontal strand and pushing it once across the longitude of the solid torus. Once we isotope $\beta$ over the longitude, we follow the pattern around the meridian until we end up inside the bygon that contains the $z$ basepoint, without crossing the longitude again. A intermediate stage of this isotopy is shown in figure \ref{fingermove}. If we push the $\beta$ curve over the $z$ basepoint, we arrive at the pattern shown in figure \ref{trefoilpattern $P_3$}, which we will denote by $P_{3,1}$. 

In general, we take the bottom most horizontal strand of the $\beta$ curve in the genus-1 doubly-pointed bordered Heegaard diagram for the $(p,1)$-cable pattern, push it once over the longitude, and then follow the pattern around the meridian until we end up inside the bygon that contains the $z$ basepoint. If we push the $\beta$ curve over the basepoint, we arrive at a $(1,1)$ diagram for a pattern that we denote $P_{p,1}$. The lift of the pattern $P_{p,1}$ is shown in figure \ref{tauepsilonposgeneral}, where we see that it looks like the lift of the $(p,1)$ -cable pattern with one extra arm. By construction, since we only crossed the longitude $\lambda$ once in our isotopy, we increased the number of intersections with the longitude by two. Therefore $\rk(\HFKhat(S^3,P_{p,1}(U)))=3$ for all $p>1$. Alternatively, pairing this pattern with $\CFDhat(S^3\setminus U)$ (whose immersed curve is a horizontal line) results in three intersection points and no differentials. As the rank of knot Floer homology detects the trefoil knot \cite{geography}*{Corollary 8}, we know that $P_{p,1}(U)$ has the knot type of the trefoil in $S^3$. As mentioned in the introduction, we will call such a pattern $P \subset S^1 \times D^2$ a \textit{trefoil pattern}. In the next section we will use the procedure described in section \ref{tauofpairing} to prove Theorem \ref{tauoftrefoil}.

\begin{figure}[!tbp]
  \centering
  \begin{minipage}[b]{0.3\textwidth}
  \includegraphics[width=1.2\textwidth]{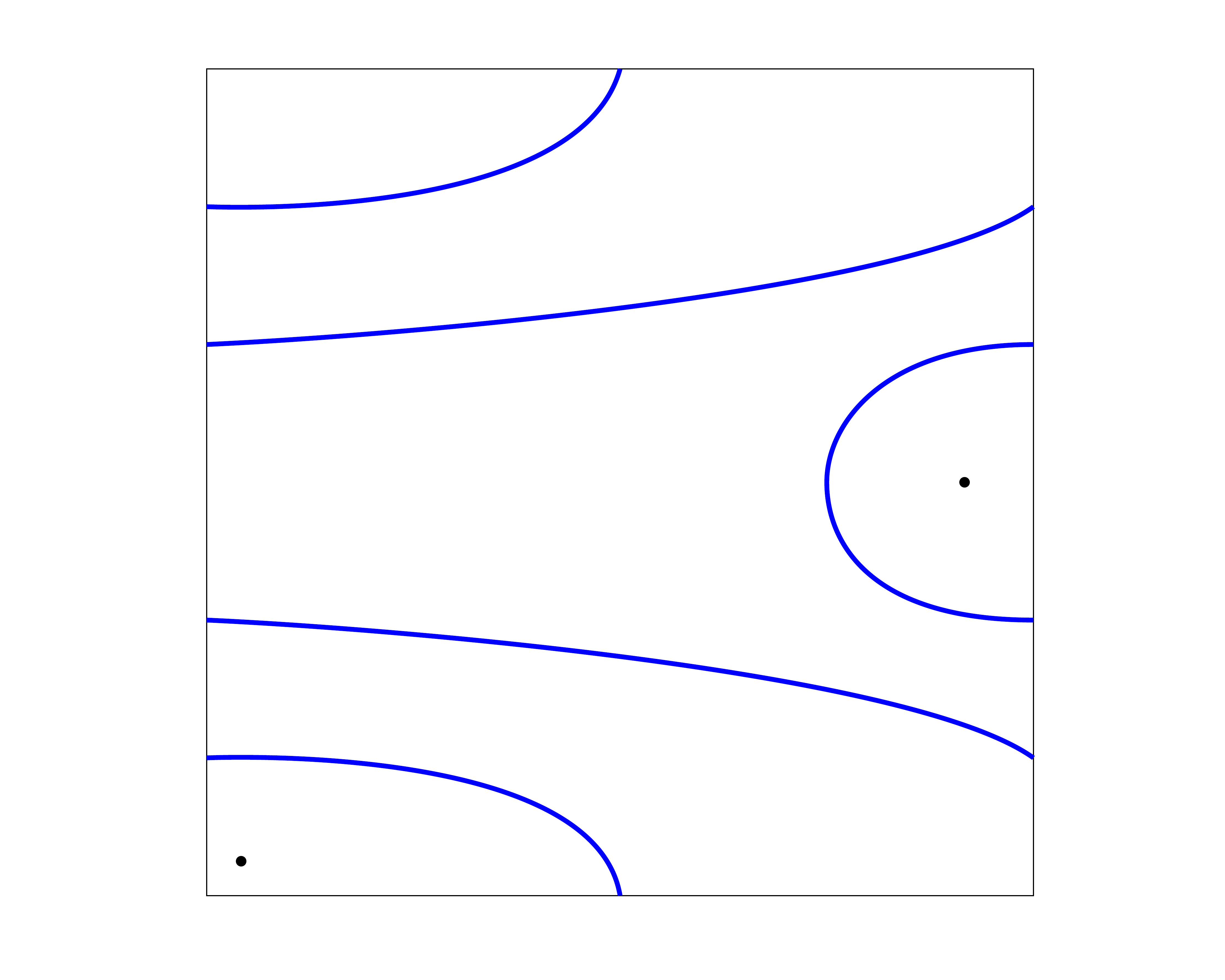}
    \caption{Doubly pointed Heegaard diagram for $(3,1)$ cable pattern}\label{(3,1)cable}
  \end{minipage}
  \hspace{2in}
  \begin{minipage}[b]{0.3\textwidth}
  \begin{tikzpicture}
\node[anchor=south west,inner sep=0] at (0,0)    {\includegraphics[width=1.2\textwidth]{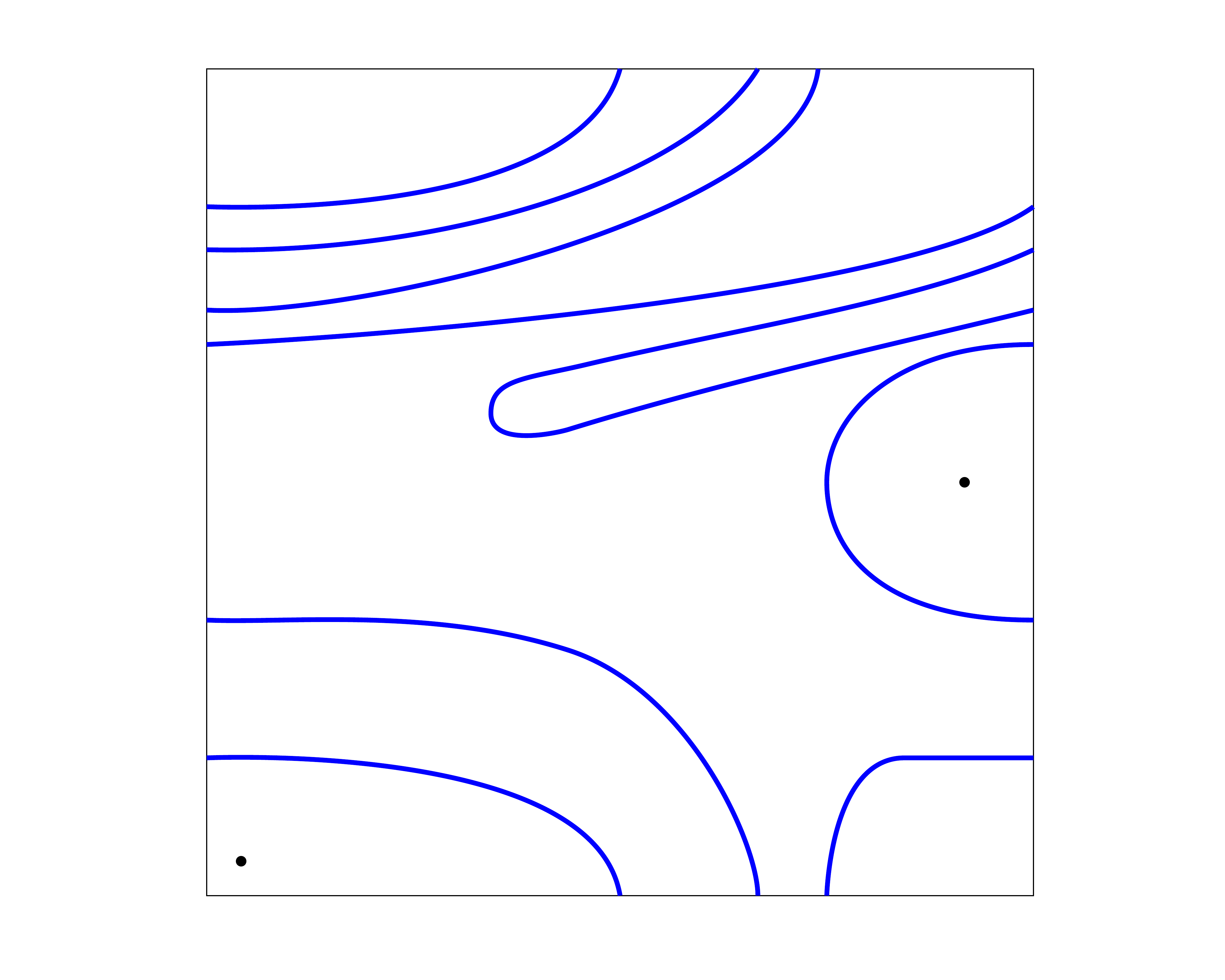}};
\end{tikzpicture}
    \caption{Midway through the isotopy}\label{fingermove}
  \end{minipage}
  \end{figure}
  
\begin{figure}[!tbp]
  \centering
  \begin{minipage}[b]{0.3\textwidth}
  \begin{tikzpicture}
\node[anchor=south west, inner sep=0] at (0,0)  {\includegraphics[width=1.2\textwidth]{trefoilfrom31cable}};
\node at (3,4) {$\lambda$};
\node at (.5,2) {$\mu$};
\node at (1.2,.5) {$w$};
\node at (4.15,2) {$z$};
\end{tikzpicture}
    \caption{Doubly pointed bordered Heegaard diagram for the trefoil pattern $P_{3,1}$}\label{trefoilpattern $P_3$}
  \end{minipage}
  \hspace{2in}
  \begin{minipage}[b]{0.3\textwidth}
  \begin{tikzpicture}
\node[anchor=south west,inner sep=0] at (0,0)    {\includegraphics[width=1.2\textwidth]{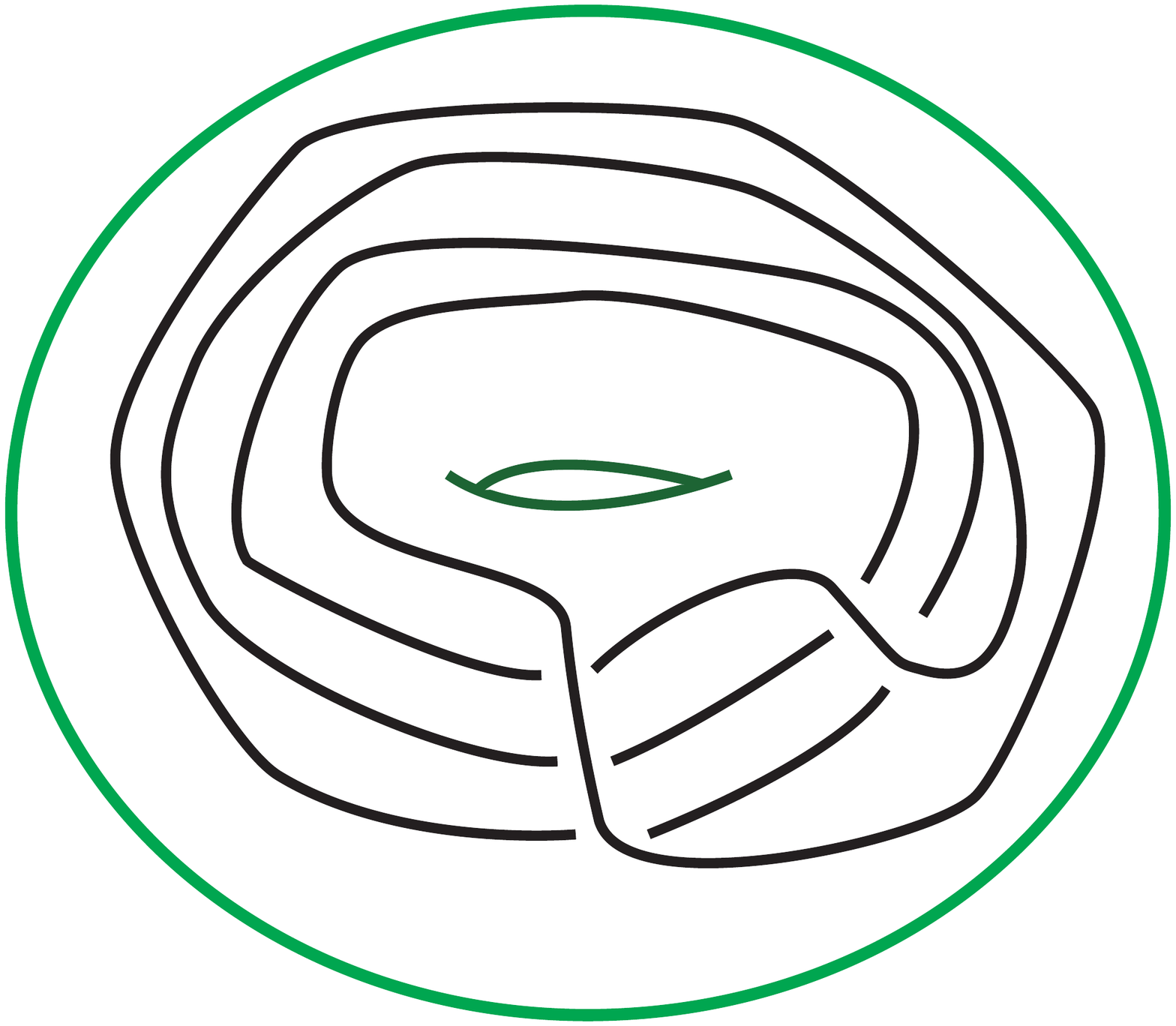}};
\end{tikzpicture}
\vspace{-.7in}
    \caption{The trefoil pattern $P_{3,1}$ in the solid torus}\label{trefoilpatternintorus}
  \end{minipage}
\end{figure}


\subsection{$\tau$ of $0$ framed satellites with arbitrary companions}
In the previous section we constructed, for each $p>1$, a trefoil pattern in the solid torus. It follows from \cite[Lemma 6.3]{Chen} that $w(P_{p,1})=p+1$. The pattern $P_{3,1}$ is shown in the solid torus in figure \ref{trefoilpatternintorus}. In this section we show how to compute $\tau(P_{p,1}(K))$ for $K$ an arbitrary knot in $S^3$. As we will see, the answer only depends on the values of $\tau(K)$ and $\epsilon(K)$.

\begin{proof}[Proof of Theorem \ref{tauoftrefoil}] 

We begin with a discussion of how to determine the absolute Alexander grading of intersection points representing generators of $\HFKhat(S^3,P_{p,1}(K))$ in the pairing diagrams in figures \ref{tauepsilonposgeneral} and \ref{taunegepsilongeneral}. For example, in figure \ref{tauepsilonposgeneral} we see a lift of the $\beta$ curve to the universal cover. The dotted portions of the $\beta$ curve represents that the $\beta$ curve crosses $p-3$ columns that are not drawn, and the $\beta$ curve is completely horizontal. If we focus in on one row, for example the row \textcolor{orange}{highlighted} in figure \ref{tauepsilonposgeneral}, we can determine the relative Alexander grading of all the intersection points by Lemma \ref{Alex}. We then determine the relative Alexander grading of all the other generators by noting that by \cite{Chen}*{Lemma 6.3}, if $x$ and $x'$ are intersection points that occur on arcs of the $\beta$ curve that differ by a meridional deck transformation (shifting the picture in the universal cover down a row), then their Alexander grading difference is $w(P)$, where $w(P)$ denotes the winding number of the pattern. For example in figure \ref{tauepsilonposgeneral} the intersection points $x$ and $x'$ lie on arcs of the $\beta$ curve that are related by a meridional deck transformation. It is easy to see that $A(x)-A(x')=p+1=w(P)$. Now, to determine the absolute Alexander grading, note that the conjugation symmetry of knot Floer homology is witnessed in the pairing diagram by the hyperelliptic involution. That is, if we rotate the entire picture by $\pi$, and exchange the $z$ and $w$ basepoints, we will get the same complex. Therefore, if any intersection is fixed under this involution then it must have Alexander grading zero. In particular, we can see that the intersection that is fixed will occur along the arc of the $\beta$ curve in figure \ref{tauepsilonposgeneral} that contains the point labelled $x$. Since all intersections along this arc will have Alexander grading zero by Lemma \ref{Alex} it is enough to compute Alexander grading relative to any intersection between $\alpha(K)$ and $\beta$ that lies on this arc. From now on, we assume that this has been done and the Alexander gradings that appear in figure \ref{tauepsilonposgeneral} are absolute and not relative. 

Now, we turn to discuss how we determine which intersection point survives the $z$ basepoint Alexander grading spectral sequence. By an isotopy of the $\beta$ curve only crossing $z$ basepoints, we can isotope $\beta$ to the light blue curve in figures \ref{tauepsilonposgeneral} and \ref{taunegepsilongeneral} in the far right of the diagram. From this observation, we see that there is a choice of cancelling disks in the pairing diagram so that when we run the spectral sequence, the last remaining intersection point lies in the far right column of the pairing diagram. This choice of cancelling disks echos the choice made in the example shown in figure \ref{P3pairingRHTwithbouys} above. Next, recall that the immersed curve $\alpha(K)$, for a general knot $K$, consists of two kinds of components. There is the essential curve component $\gamma_0$ with no non-trivial local systems which wraps around the longitude of the torus and there are (potentially) other components that are immersed with local systems which all lie in a neighborhood of the meridian. Since $\beta$ can be isotoped away from a neighborhood of the meridian by crossing only $z$ basepoints, the intersection point that survives the Alexander filtration spectral sequence is an intersection between the essential component $\gamma_0$ and $\beta(P)$. Therefore, since the essential curve component has the form described in Lemma \ref{essentialcomponent} and depends only on the values of $\tau(K)$ and $\epsilon(K)$, it remains to analyse the following cases to determine the absolute Alexander grading of the generator that survives. \\

\noindent\underline{\textbf{$\tau(K)>0, \epsilon(K)=1$}}: In this case, the part of the essential component of the immersed curve for $K$ coming from the unstable chain slopes upward for $2\tau(K)$ rows and turns down at the top and up at the bottom. See figure \ref{P3pairingRHT} for an example when $p=3$ and $K=T_{2,3}$ and figure \ref{tauepsilonposgeneral} for the general case, where in that figure, we pay attention to the piece of the essential component that is \textcolor{green}{dotted} and we only draw the portion of the essential component of $\alpha(K)$ that carries the intersecion that survives the spectral sequence. In this case we see that the surviving intersection point is the one labelled $a$ in figure \ref{tauepsilonposgeneral}. To compute what this Alexander grading is, we use Lemma \ref{Alex}. When we follow the $\beta$ curve from the generator with Alexander grading $0$, labelled $x$ in the figure, we travel down $\tau(K)$ rows and then cross one extra $\delta_{w,z}$ arc. The $\tau(K)$ rows results in a change in Alexander filtration by $w(P)\tau(K)=(p+1)\tau(K)$, and crossing one more $\delta_{w,z}$ arc gives the result:

$$\tau(P_{p,1}(K))=(p+1)\tau(K)+1.$$\\

For example, in figure \ref{P3pairingRHT}, we saw earlier that $\tau(P_{3,1}(T_{2,3}))=5$. 


\begin{figure}[!tbp]
  \centering
  \begin{minipage}[b]{0.3\textwidth}
  \begin{tikzpicture}
\node[anchor=south west,inner sep=0] at (0,0)    {\includegraphics[width=1.2\textwidth]{P3trefoilpatternpairing}};

\node[font=\tiny] at (2.82,3.5) {$x$};
\node at (2.9,3.45) {$\tikzcirclee{1pt}$};
\node at (4.29,2.45) {$\tikzcirclee{1pt}$};
\node[font=\tiny] at (4.35,2.35) {$y$};

\node[font=\tiny] at (1,1.4) {$w$};
\node[font=\tiny] at (1,2.75) {$w$};
\node[font=\tiny] at (1,3.9) {$w$};
\node[font=\tiny] at (1,5.2) {$w$};

\node[font=\tiny] at (1.5,5.3) {$\delta_{w,z}$};
\node[font=\tiny] at (2.5,5.3) {$\delta_{w,z}$};
\node[font=\tiny] at (3.5,5.3) {$\delta_{w,z}$};
\node[font=\tiny] at (4.5,5.3) {$\delta_{w,z}$};

\node[font=\tiny] at (1.5,4) {$\delta_{w,z}$};
\node[font=\tiny] at (1.5,2.8) {$\delta_{w,z}$};
\node[font=\tiny] at (1.5,1.5) {$\delta_{w,z}$};

\node[font=\tiny] at (1.4,2.1) {$z$};
\node[font=\tiny] at (1.35,3.25) {$z$};
\node[font=\tiny] at (1.35,4.5) {$z$};
\node[font=\tiny] at (1.35,5.6) {$z$};

\node[font=\tiny] at (2.4,2) {$z$};
\node[font=\tiny] at (2.35,3.25) {$z$};
\node[font=\tiny] at (2.35,4.45) {$z$};
\node[font=\tiny] at (2.35,5.6) {$z$};

\node[font=\tiny] at (3.5,2) {$z$};
\node[font=\tiny] at (3.45,3.25) {$z$};
\node[font=\tiny] at (3.45,4.45) {$z$};
\node[font=\tiny] at (3.45,5.6) {$z$};

\node[font=\tiny] at (4.5,2) {$z$};
\node[font=\tiny] at (4.45,3.25) {$z$};
\node[font=\tiny] at (4.45,4.45) {$z$};
\node[font=\tiny] at (4.45,5.6) {$z$};

\node[font=\tiny] at (2,1.4) {$w$};
\node[font=\tiny] at (2.2,2.75) {$w$};
\node[font=\tiny] at (2,3.9) {$w$};
\node[font=\tiny] at (2,5.2) {$w$};

\node[font=\tiny] at (3,1.4) {$w$};
\node[font=\tiny] at (3.2,2.75) {$w$};
\node[font=\tiny] at (3,3.9) {$w$};
\node[font=\tiny] at (3,5) {$w$};

\node[font=\tiny] at (4,1.4) {$w$};
\node[font=\tiny] at (4.05,2.5) {$w$};
\node[font=\tiny] at (4,3.8) {$w$};
\node[font=\tiny] at (4,5) {$w$};
\end{tikzpicture}
    \caption{The lift of the trefoil pattern $P_3$ shown in figure \ref{trefoilpattern $P_3$} paired with the right handed trefoil. We have $\tau(P_{3,1}(T_{2,3}))=A(y)=5$.}\label{P3pairingRHT}
  \end{minipage}
  \hspace{2in}
  \begin{minipage}[b]{0.3\textwidth}
  \begin{tikzpicture}
\node[anchor=south west,inner sep=0] at (0,0)    {\includegraphics[width=1.2\textwidth]{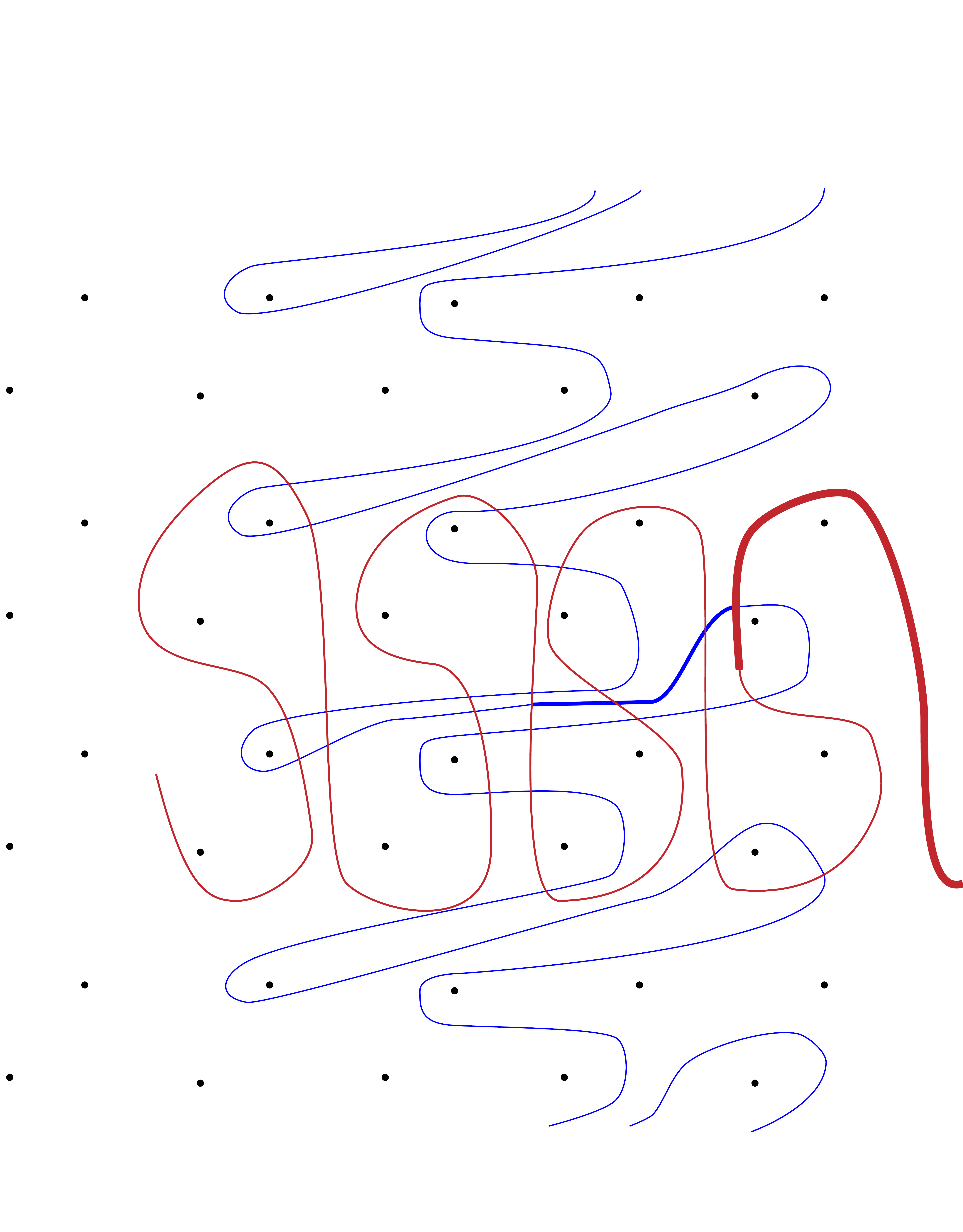}};
\node[font=\tiny] at (2.8,2.75) {$\tikzcirclee{1pt}$};
\node[font=\tiny] at (2.7,2.65) {$x$};
\node[font=\tiny] at (3.85,3.25) {$\tikzcirclee{1pt}$};
\node[font=\tiny] at (4,3.35) {$y$};
\end{tikzpicture}
    \caption{The lift of the trefoil pattern $P_3$ shown in figure \ref{trefoilpattern $P_3$} paired with the left handed trefoil. We find $\tau(P_{3,1}(T_{2,-3}))=A(y)=0$.}\label{P3pairingLHT}
  \end{minipage}
\end{figure}


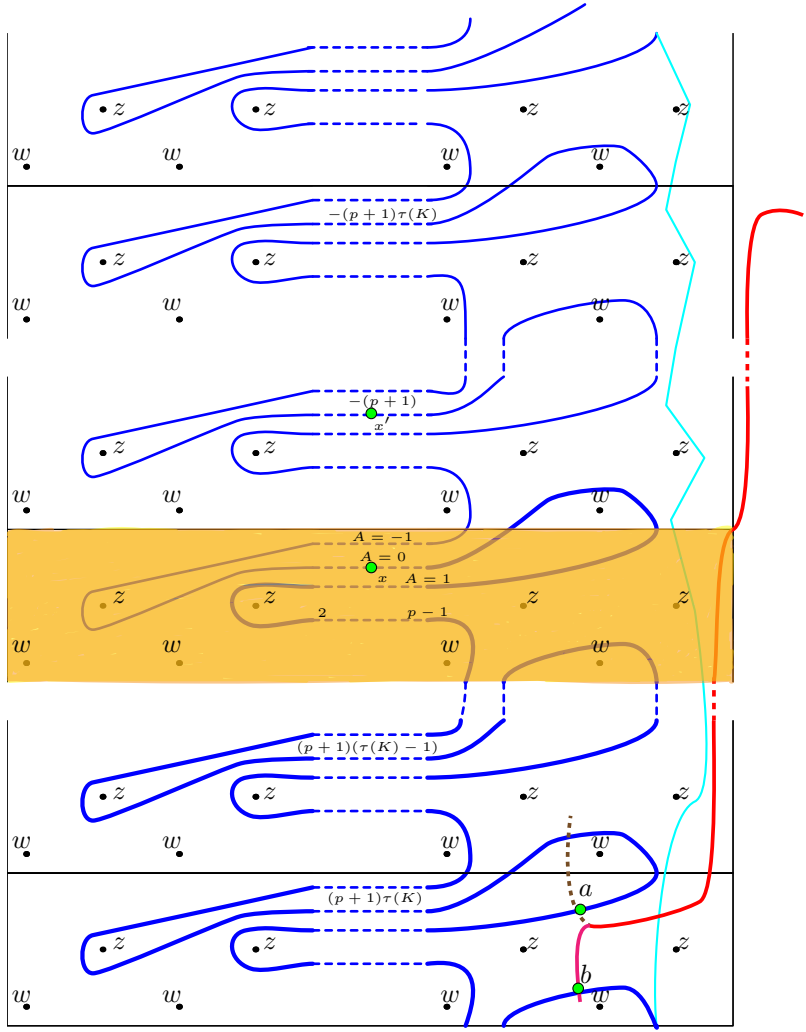
\begin{figure}
\begin{center}
\begin{tikzpicture}
\node[anchor=south west, inner sep=0] at (0,0) {\includegraphics[scale=.2]{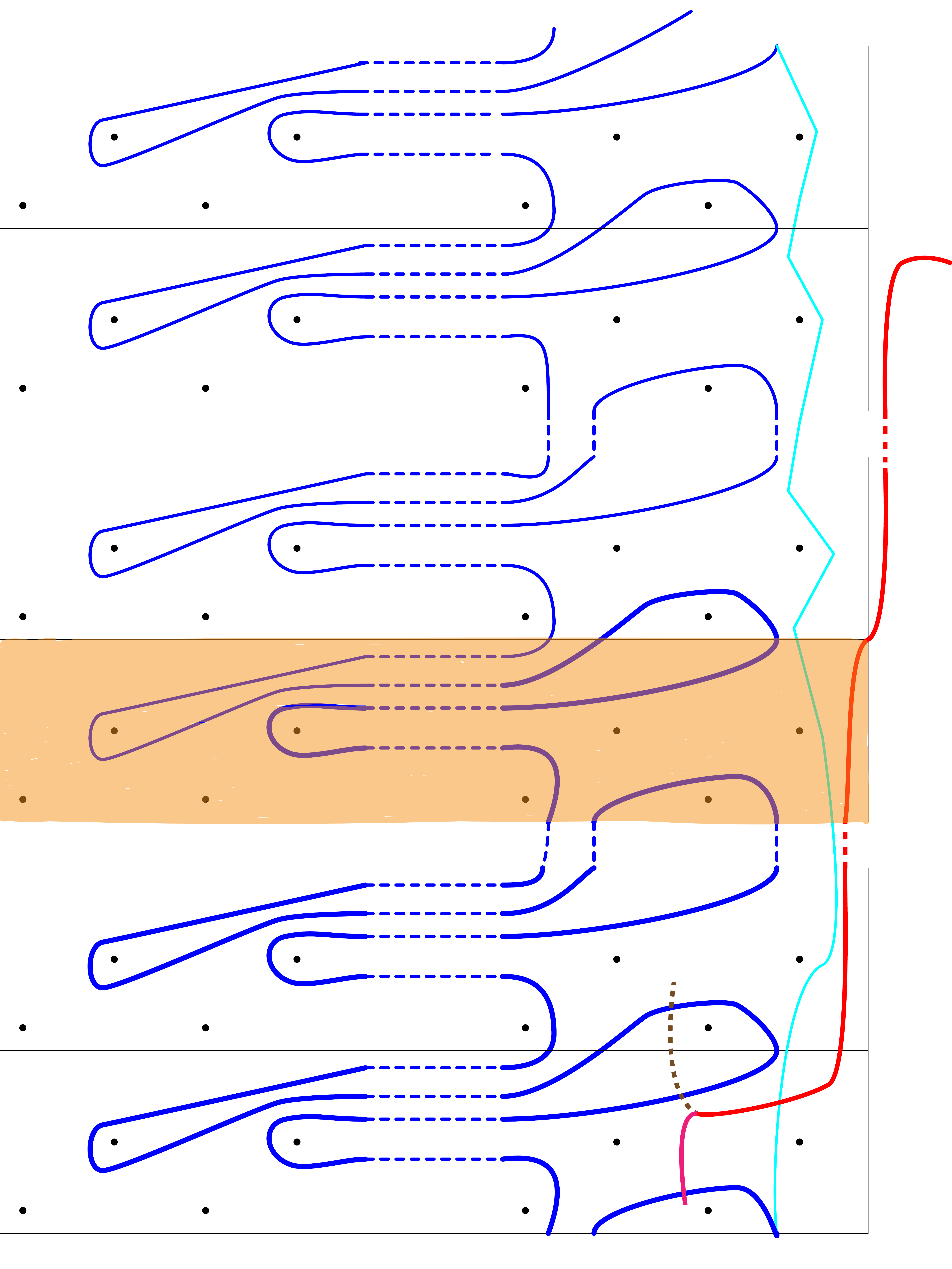}};
\node at (.2,.8) {$w$};
\node at (2.2,.8) {$w$};
\node at (5.9,.8) {$w$};
\node at (7.9,.8) {$w$};
\node at (.2,2.85) {$w$};
\node at (2.2,2.85) {$w$};
\node at (5.9,2.85) {$w$};
\node at(7.9,2.85){$w$};
\node at (.2,5.45) {$w$};
\node at (2.2,5.45){$w$};
\node at (5.9,5.45) {$w$};
\node at(7.9,5.45) {$w$};
\node at (.2,7.45) {$w$};
\node at (2.2,7.45){$w$};
\node at (5.9,7.45) {$w$};
\node at (7.9,7.45) {$w$};
\node at (.2,10) {$w$};
\node at (2.2,10) {$w$};
\node at (5.9,10) {$w$};
\node at (7.9,10) {$w$};
\node at (.2,12) {$w$};
\node at (2.2,12) {$w$};
\node at (5.9,12) {$w$};
\node at (7.9,12) {$w$};
\node at (1.5,1.5) {$z$};
\node at (3.5,1.5) {$z$};
\node at (7,1.5) {$z$};
\node at (9,1.5) {$z$};
\node at (1.5,3.5) {$z$};
\node at (3.5,3.5) {$z$};
\node at (7,3.5) {$z$};
\node at (9,3.5) {$z$};
\node at (1.5,6.1) {$z$};
\node at (3.5,6.1) {$z$};
\node at (7,6.1) {$z$};
\node at (9,6.1) {$z$};
\node at (1.5,8.1) {$z$};
\node at (3.5,8.1) {$z$};
\node at (7,8.1) {$z$};
\node at (9,8.1) {$z$};
\node at (1.5,10.6) {$z$};
\node at (3.5, 10.6) {$z$};
\node at (7,10.6) {$z$};
\node at (9,10.6){$z$};
\node at (1.5,12.6) {$z$};
\node at (3.5,12.6){$z$};
\node at (7,12.6) {$z$};
\node at (9,12.6) {$z$};

\node[font=\tiny] at (5,6.65) {$A=0$};
\node[font=\tiny] at (5.6,6.35) {$A=1$};
\node[font=\tiny] at (4.2,5.9) {$2$};
\node[font=\tiny] at (5.6,5.9) {$p-1$};
\node[font=\tiny] at (5,6.9) {$A=-1$};
\node[font=\tiny] at (5,8.7) {$-(p+1)$};
\node[font=\tiny] at (4.85,8.55) {$\tikzcirclee{2pt}$};
\node[font=\tiny] at (5,8.42) {$x'$};
\node[font=\tiny] at (4.8,4.1) {$(p+1)(\tau(K)-1)$};
\node[font=\tiny] at (4.9,2.1) {$(p+1)\tau(K)$};
\node[font=\tiny] at (5,11.2) {$-(p+1)\tau(K)$};

\node[font=\tiny] at (5,6.35) {$x$};
\node at (4.85,6.5) {$\tikzcirclee{2pt}$};

\node at (7.7,2.2) {$a$};
\node at (7.63,1.95) {$\tikzcirclee{2pt}$};

\node at (7.7,1.1) {$b$};
\node at (7.6,.9) {$\tikzcirclee{2pt}$};
\end{tikzpicture}
\caption{The general case with $\tau(K)>0$ and $\epsilon(K)=\pm1$. $\epsilon(K)=1$ is shown as a \textcolor{green}{dotted} arc, and $\epsilon(K)=-1$ is shown as a \textcolor{pink}{solid} arc}\label{tauepsilonposgeneral}
\end{center}
\end{figure}

\noindent\underline{\textbf{$\tau(K)>0, \epsilon(K)=-1$}}: In this case, that part of the essential component of the immersed curve for $K$ slopes upward for $2\tau(K)$ rows, but it turns down at the bottom and up at the top. See figure \ref{tauepsilonposgeneral}, where we pay attention to the \textcolor{pink}{solid} portion of the essential component of the immersed curve in the bottom right that turns down and contains the intersection point labelled $b$. Since $b$ is the only intersection point remaining after isotoping $\beta$ to the light blue curve, we see that $A(b)=\tau(P_{p,1}(K))$. Using figure \ref{tauepsilonposgeneral}, we see that this intersection point occurs exactly $\tau(K)+1$ rows below the generator with Alexander grading zero. Using Lemma \ref{Alex}, we see that

$$\tau((P_{p,1})(K))=(p+1)(\tau(K)+1)$$


\noindent\underline{\textbf{$\tau(K)<0, \epsilon(K)=1$}}: In this case, the essential component of the immersed curve $\alpha$ slopes downward for $2\tau(K)$ rows, and turns up at the top and down at the bottom, see figure \ref{taunegepsilongeneral} where in the case $\epsilon(K)=1$, we focus on the \textcolor{orange}{solid} portion of the essential component of $\alpha(K)$ in the upper right, which contains the intersection point labeled $a$. Just as in the previous cases we see that $a$ survives the $z$ basepoint spectral sequence and we find the Alexander grading of $a$ by counting how may rows above the generator with Alexander grading zero this intersection point lies. From figure \ref{taunegepsilongeneral} we see that the intersection point $y$ lives exactly $\tau(K)$ rows above the intersection point $x$ with $A(x)=0$. Therefore, by Lemma \ref{Alex} $A(y)=w(P)\tau(K)$. Then, we see that $A(a)-A(y)=1$, so we have

$$A(a)=\tau(P_{p,1}(K))=(p+1)\tau(K)+1.$$

\begin{figure}
\begin{center}
\begin{tikzpicture}
\node[anchor=south west, inner sep=0] at (0,0) {\includegraphics[scale=.2]{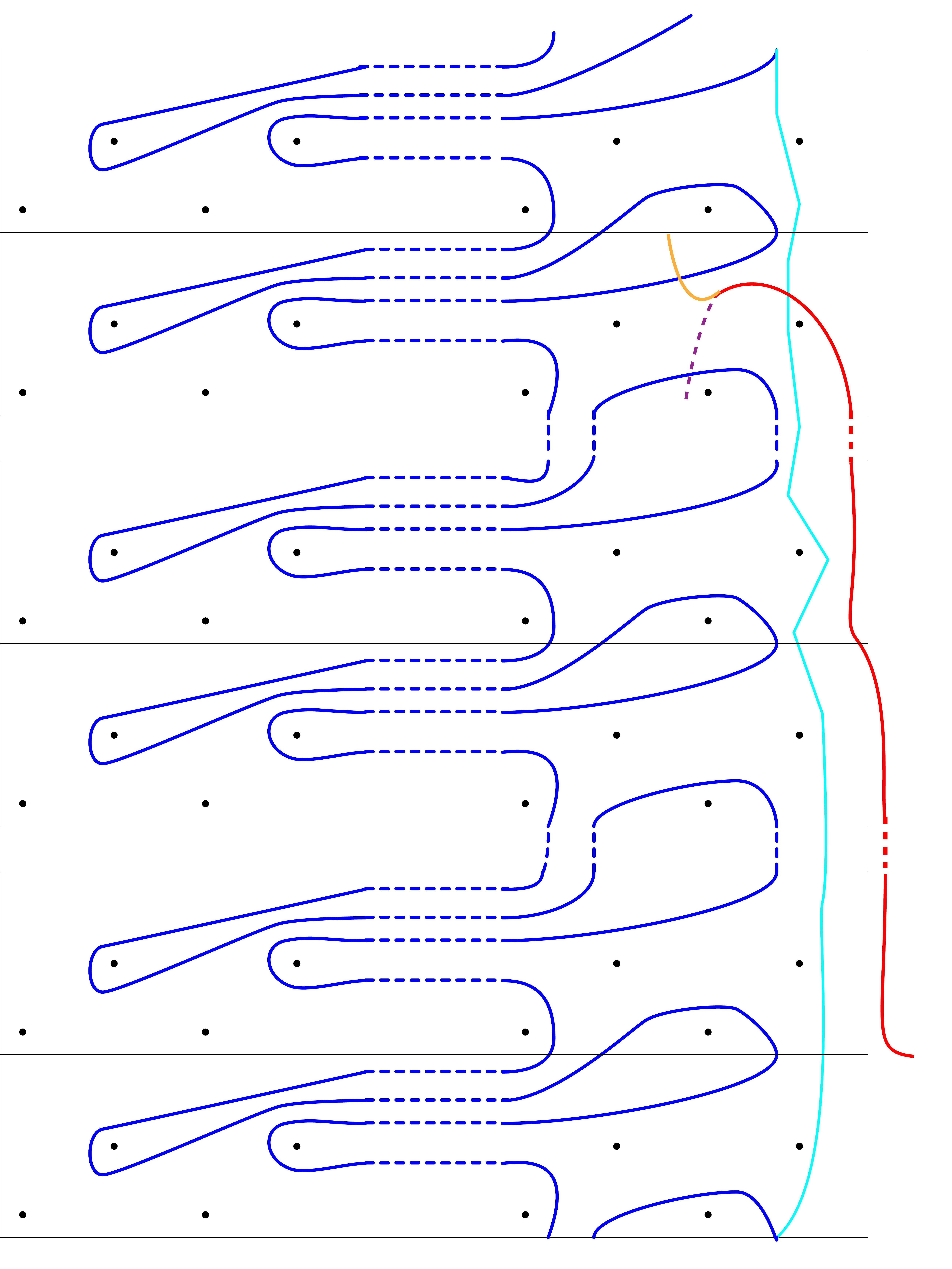}};
\node at (.2,.8) {$w$};
\node at (2.2,.8) {$w$};
\node at (5.9,.8) {$w$};
\node at (7.9,.8) {$w$};
\node at (.2,2.85) {$w$};
\node at (2.2,2.85) {$w$};
\node at (5.9,2.85) {$w$};
\node at(7.9,2.85){$w$};
\node at (.2,5.45) {$w$};
\node at (2.2,5.45){$w$};
\node at (5.9,5.45) {$w$};
\node at(7.9,5.45) {$w$};
\node at (.2,7.45) {$w$};
\node at (2.2,7.45){$w$};
\node at (5.9,7.45) {$w$};
\node at (7.9,7.45) {$w$};
\node at (.2,10) {$w$};
\node at (2.2,10) {$w$};
\node at (5.9,10) {$w$};
\node at (7.9,10) {$w$};
\node at (.2,12) {$w$};
\node at (2.2,12) {$w$};
\node at (5.9,12) {$w$};
\node at (7.9,12) {$w$};
\node at (1.5,1.5) {$z$};
\node at (3.5,1.5) {$z$};
\node at (7,1.5) {$z$};
\node at (9,1.5) {$z$};
\node at (1.5,3.5) {$z$};
\node at (3.5,3.5) {$z$};
\node at (7,3.5) {$z$};
\node at (9,3.5) {$z$};
\node at (1.5,6.1) {$z$};
\node at (3.5,6.1) {$z$};
\node at (7,6.1) {$z$};
\node at (9,6.1) {$z$};
\node at (1.5,8.1) {$z$};
\node at (3.5,8.1) {$z$};
\node at (7,8.1) {$z$};
\node at (9,8.1) {$z$};
\node at (1.5,10.6) {$z$};
\node at (3.5, 10.6) {$z$};
\node at (7,10.6) {$z$};
\node at (9,10.6){$z$};
\node at (1.5,12.6) {$z$};
\node at (3.5,12.6){$z$};
\node at (7,12.6) {$z$};
\node at (9,12.6) {$z$};

\node at (5,6.35) {$x$};
\node at (4.85,6.45) {$\tikzcirclee{2pt}$};
\node[font=\tiny] at (5,6.555) {$A=0$};

\node[font=\tiny] at (5,8.6) {$A=-(p+1)$};

\node[font=\tiny] at (4.9,4) {$A=(p+1)(|\tau(K)|-1)$};
\node[font=\tiny] at (5,2) {$A=(p+1)|\tau(K)|$};
\node[font=\tiny] at (5,11.2) {$A=(p+1)\tau(K)$};

\node at (7.55,11) {$\tikzcirclee{2pt}$};
\node at (7.6,11.2) {$a$};
\node at (5,10.85) {$y$};
\node at (4.85,11.05) {$\tikzcirclee{2pt}$};

\node at (7.5,10.1) {$b$};
\node at (7.72,10) {$\tikzcirclee{2pt}$};
\end{tikzpicture}
\caption{The general case with $\tau(K)<0$ and $\epsilon(K)=\pm1$. $\epsilon(K)=-1$ is shown as a \textcolor{purple}{dotted} arc and $\epsilon(K)=1$ is shown as a \textcolor{orange}{solid} arc}\label{taunegepsilongeneral}
\end{center}
\end{figure}
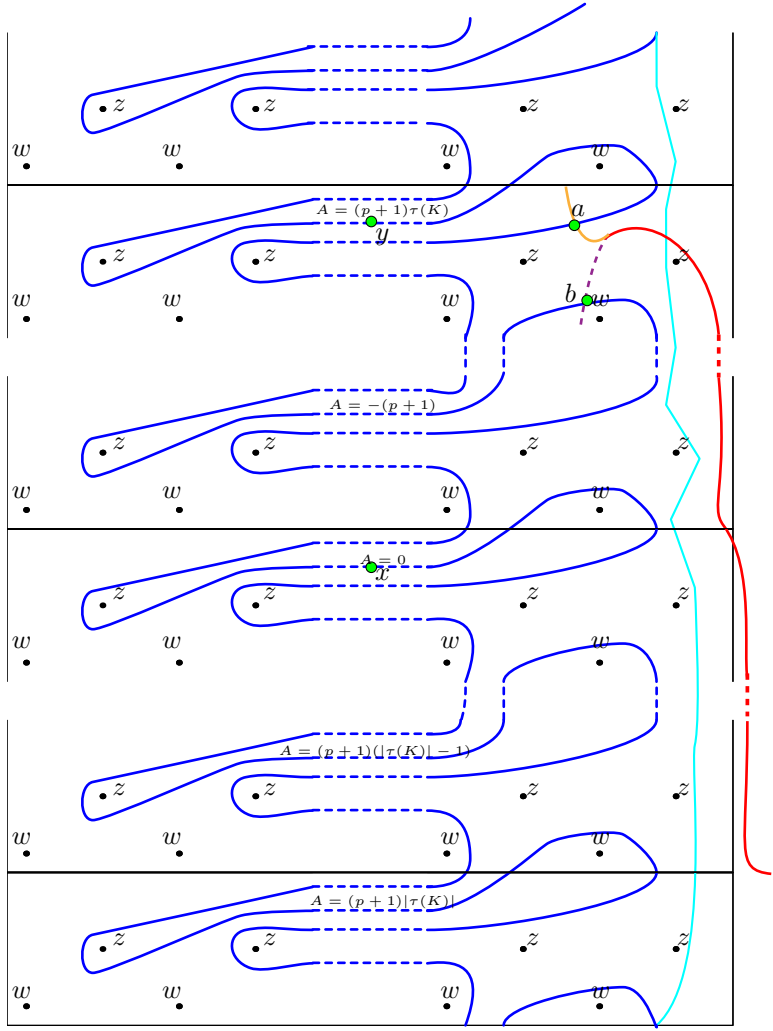

\noindent\underline{\textbf{$\tau(K)<0, \epsilon(K)=-1$}}: This case is similar to the previous cases. Here the relevant portion of the $\alpha$ immersed curve slopes downward and turns down at the top and up at the bottom, see figure \ref{taunegepsilongeneral} paying attention to the \textcolor{purple}{dotted} portion of the curve in the upper right. The intersection point labelled $b$ is the one that survives the $z$-basepoint spectral sequence. We count the number of rows above the central row that this intersection point occurs to compute $A(b)$. The result is

$$\tau(P_{p,1}(K))=(p+1)(\tau(K)+1).$$

For an example, consider figure \ref{P3pairingLHT}. We see that the intersection point that survives the $z$ basepoint spectral sequence lies on both the bold portion of the $\alpha$ curve and the bold portion of the $\beta$ curve. It is easy to see from the picture that $\tau(P_{3,1}(T_{2,-3}))=0$, since travelling along the bold potion of the $\beta$ curve, we do not cross any $\delta_{w,z}$ arcs.\\

\noindent\underline{\textbf{$\epsilon(K)=0$}}: In this case, we also have $\tau(K)=0$. Hence the essential curve component is horizontal. Therefore, the intersection point that survives the $z$ basepoint spectral sequence has alexander grading $1$, which is the same as $\tau(T_{2,3})$ as expected. \qedhere

\end{proof}

\section{Three genus and fiberedness}\label{threegenusandfibered}

In this section we will prove theorem \ref{fiberedpattern} from the introduction. Recall that the knot Floer homology detects both the three-genus and the fiberedness of a knot $K \subset S^3$ in the following sense. The genus of a knot is the largest Alexander grading supporting non-zero Floer homology by \cite{genusdetection}. Further, the knot is fibered if and only if the knot Floer homology is one dimensional in this top Alexander grading by \cite{surfacedecomp}. 

Recall from \cite{fiberedsatellite} that, for a non-trivial companion knot, the satellite knot $P(K)$ is fibered if and only in the companion knot $K$ is fibered in $S^3$ and the pattern is fibered in the solid torus. Therefore, to prove theorem \ref{fiberedpattern}, it is enough to show that $P_{p,1}(T_{2,3})$ is fibered. Furthermore, for a non trivial knot $K$, we have the classical genus of a satellite formula

\begin{equation}\label{genussatellite}g(P(K))=|w(P)|g(K)+g(P).\end{equation} So, to compute $g(P)$ it is enough to compute $g(P(T_{2,3}))$.

\begin{proof}[Proof of Lemma \ref{patterngenus}]

We will make use of the pairing diagram in figure \ref{generalP_pRHT} which computes $\HFKhat(S^3,P_{p,1}(T_{2,3}))$. In the diagram, we see that the generator $a$ has the largest Alexander grading of any intersection point, and we compute using Lemma \ref{Alex} that $A(a)=p+2$. Hence $g(P_{p,1}(T_{2,3}))=p+2$. Then using equation \ref{genussatellite} we have

$$p+2=g(P_{p,1}(T_{2,3}))=(p+1)g(T_{2,3})+g(P)=p+1+g(P),$$ 
which implies that $g(P)=1$. 

\end{proof}

With Lemma \ref{patterngenus} in hand, we can prove that the for all $p>1$, patterns $P_{p,1}$ are fibered. 

\begin{proof}[Proof of Theorem \ref{fiberedpattern}]

Since knot Floer homology detects fibered knots, and a satellite knot is fibered if and only if the pattern and comanion are fibered, to show that the pattern $P_{p,1}$ is fibered it is enough to show that $\rk(\HFKhat(S^3,P_{p,1}(T_{2,3}),p+2))=1$ for all $p>1$. To this end, consider the pairing diagram for $\HFKhat(S^3,P_{p,1}(T_{2,3}))$. By the proof of Lemma \ref{patterngenus}, we know that $a$ has the largest Alexander grading of any intersection point. To show that $P_{p,1}(K)$ is fibered, we will show that for any other intersection point $x$ in the pairing diagram, we have $A(x)<A(a)$. To this end, note that the Alexander grading is weakly decreasing as we travel up the pairing diagram on the $\beta$ curve. It follows from lemma \ref{Alex} that $A(a)-A(b)=1$ and that $A(x) \leq A(b)$ for any other intersection point $x$. Therefore $\HFKhat(S^3,P_{p,1}(T_{2,3}),p+2)$ is one dimensional for all $p>1$, and so the satellite knot $P_{p,1}(T_{2,3})$ is fibered for all $p>1$. Since a satellite knot with non trivial companion is fibered if and only if both the pattern is fibered and the companion is fibered \cite{fiberedsatellite}, it follows that $P$ is a fibered pattern.\qedhere

\end{proof}

\begin{figure}
\begin{center}
\begin{tikzpicture}
\node[anchor=south west, inner sep=0] at (0,0) {\includegraphics[scale=.3]{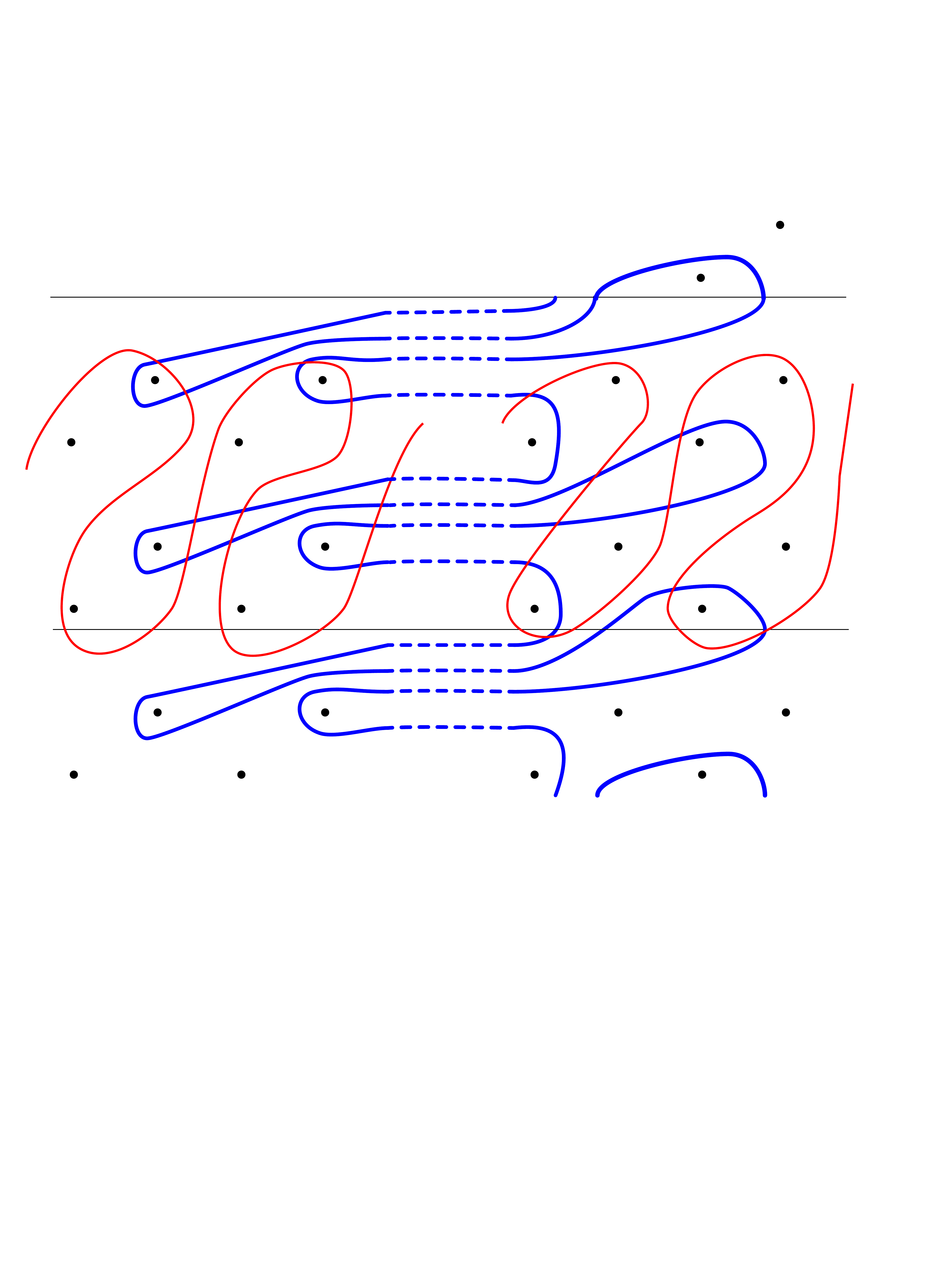}};
\node at (11,8.8) {$a$};
\node at (11,9.2) {$\tikzcirclee{2pt}$};
\node at (9.75,9.8) {$\tikzcirclee{2pt}$};
\node at (9.6,9.9) {$b$};
\node[font=\tiny] at (6.5,11.2) {$A=0$};
\node[font=\tiny] at (6.5,10.85) {$A=1$};
\node[font=\tiny] at (4.92,10.25) {$A=2$};
\node[font=\tiny] at (8.25,10.25) {$A=p-1$};
\node[font=\tiny] at (6.5,9.13) {$A=p$};
\node[font=\tiny] at (6.5,8.8) {$A=p+1$};
\node[font=\tiny] at (6.5,8.45) {$A=p+2$};
\end{tikzpicture}
\caption{The pairing diagram computing $\HFKhat(S^3,P_{p,1}(T_{2,3}))$}\label{generalP_pRHT}
\end{center}
\end{figure}

Recall from the introduction that fibered knots have unique minimal genus Seifert surfaces. Hence, for a fibered pattern $P$ and a fibered knot $K$, the satellite knot $P(K)$ also has a unique minimal genus Seifert surface. Given this, one might wonder when the operation of taking a satellite of a non-trivial knot can increase or decrease the number of non-isotopic Seifert surfaces in the knot complement. In this direction, we prove Propositions \ref{Seifertsurface} and \ref{taut} from the introduction, which imply that for knots with small rank knot Floer homology in the top Alexander grading the process of taking a satellite with a fibered pattern preserves the property of having a unique minimal genus Seifert surface as well as the property of having a depth at most one codimension one taut folitation of the complement. 

\begin{proof}[Proof of Proposition \ref{Seifertsurface}]
Suppose $K$ is a knot with $\rk(\HFKhat(S^3,K,g(K)))<4$ and $P$ is a fibered pattern. Then by \cite{Seifertsurface}*{Theorem 2.3} it follows that $K$ has a unique minimal genus Seifert surface up to isotopy. By equation \ref{nifiberedpattern} we have $\rk(\HFKhat(S^3,P(K),g(P(K))))<4$. Hence $P(K)$ also has a unique minimal genus Seifert surface by \cite{Seifertsurface}*{Theorem 2.3}. Repeating the above argument, we see that $P^i(K)$ has a unique minimal genus Seifert surface up to isotopy for all $i \geq 1$. \qedhere

\end{proof}

\begin{proof}[Proof of Proposition \ref{taut}]

Suppose $K$ is a knot with $\rk(\HFKhat(S^3,K,g(K)))=3$. So $a_{g(K)}$, the coefficient of $(t^g+t^{-g})$ in the symmetrised Alexander polynomial for $K$, is equal to $\chi(\HFKhat(S^3,K,g(K)))$ so is non zero. It follows from \cite{surfacedecomp}*{Theorem 1.8} that $S^3\setminus\nu(K)$ has a depth $\leq 1$ taut foliation transverse to $\partial(\nu(K))$. Then, equation \ref{nifiberedpattern} implies that $\rk(\HFKhat(S^3,P(K),g(P(K))))=3$, and so $a_{g(P(K))}=\chi(\HFKhat(S^3,P(K),g(P(K))))\neq 0$, where $a_{g(P(K))}$ is the analogous coefficient of the Alexnader polynomial for $P(K)$. So $S^3\setminus P(K)$ has a depth $\leq 1$ taut foliation transverse to $\partial\nu(K)$ again by \cite{surfacedecomp}*{Theorem 1.8}. Ths argument can be repeated to show that $S^3\setminus P^i(K)$ also has a depth $\leq 1$ taut foliation transverse to $\partial \nu(P^i(K))$ for all $i \geq 1$. \qedhere

\end{proof}

\section{Next to top Alexander grading}\label{nexttotop}

In this section we prove Theorem \ref{monodromy} from the introduction. First, we recall the notion of right and left veering monodromy following \cite{baldwinveering}. Suppose that $\Sigma$ is a surface with non-empty boundary and $a$ and $b$ are two properly embedded arcs in $\Sigma$. We say that $a$ is to the right of $b$ at $p$, denoted $a \geq_p b$ if $p$ is a common endpoint of both arcs and either $a$ is isotopic to $b$ rel boundary, or after isotoping $a$ rel boundary so that it intersects $b$ minimally, $a$ is to the right of $b$ in a neighborhood of $p$.
Now, suppose that $\phi: \Sigma \to \Sigma$ is a homeomorphism of $\Sigma$ which restricts to the identity on a boundary component $B$ of $\Sigma$. Then we say that $\phi$ is right veering at $B$ if
$$\phi(a)\geq_p a$$ for every properly embedded arc $a \subset \Sigma$ and every $p \in \partial a \cap B$. A map $\phi$ is called right veering if it is right veering at every boundary component of $\Sigma$. We call a map $\phi$ left veering if its inverse is right veering. 






Recall from Theorem \ref{veering} that, for a fibered knot, we can detect when the monodromy of a fibration is right or left veering by computing the next to top Alexander graded piece of the knot floer homology to be one dimensional.

\begin{figure}[!tbp]
  \centering
  \begin{minipage}[b]{0.3\textwidth}
  \begin{tikzpicture}
\node[anchor=south west,inner sep=0] at (0,0)    {\includegraphics[width=1.2\textwidth]{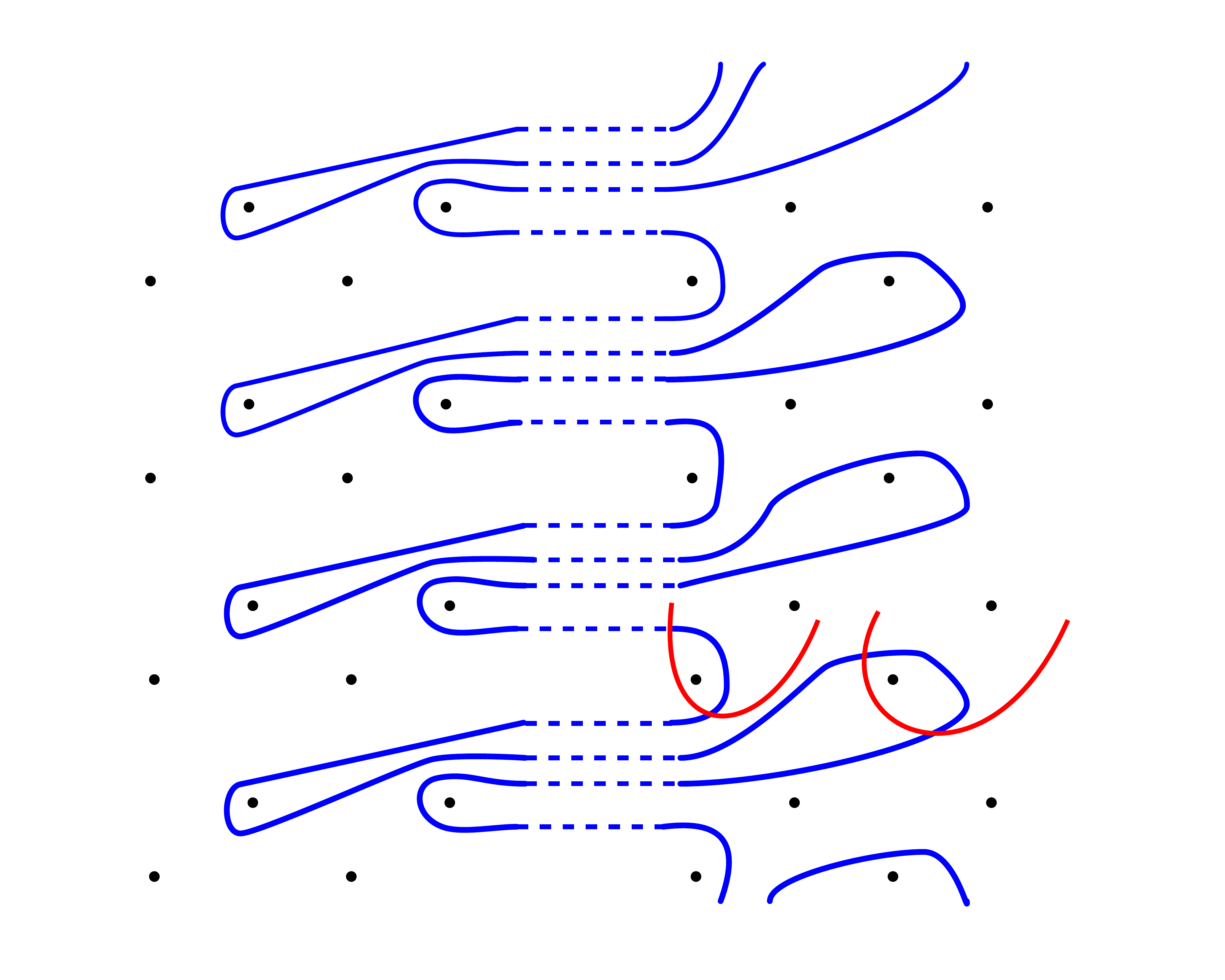}};
\node[font=\tiny] at (3.8,1) {$\tikzcirclee{1pt}$};
\node[font=\tiny] at (4,.9) {$a$};
\node[font=\tiny] at (3.55,1.35) {$\tikzcirclee{1pt}$};
\node[font=\tiny] at (3.65,1.4) {$b$};

\end{tikzpicture}
    \caption{$P_p$ paired with a fibered knot with $\tau(K)=g(K)$}\label{taupos}
  \end{minipage}
  \hspace{2in}
  \begin{minipage}[b]{0.3\textwidth}
  \begin{tikzpicture}
\node[anchor=south west,inner sep=0] at (0,0)    {\includegraphics[width=1.2\textwidth]{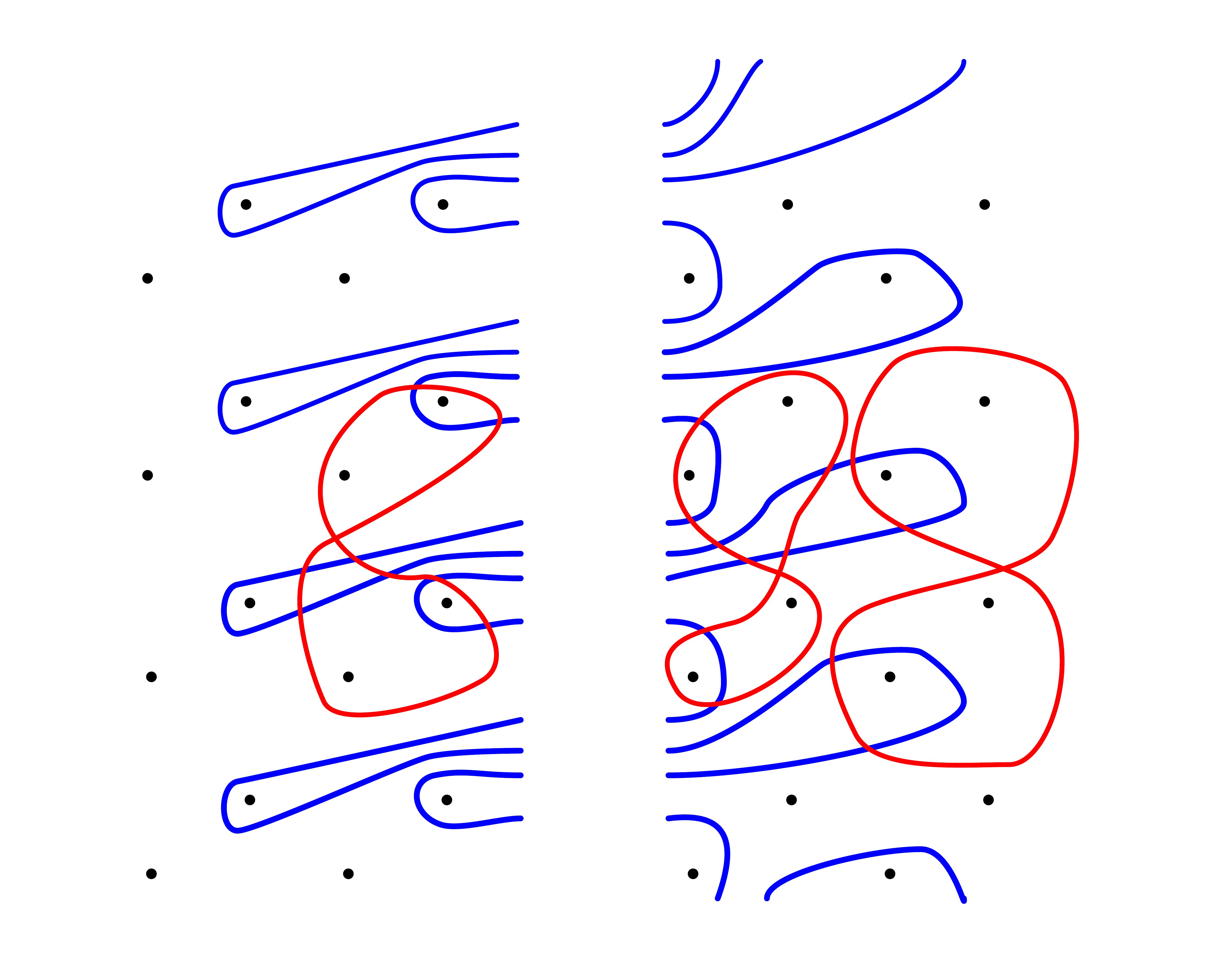}};
\node[font=\tiny] at (3.55,.95) {$\tikzcirclee{1pt}$};
\node[font=\tiny] at (3.5,.8) {$a$};
\node[font=\tiny] at (3.4,1.3) {$\tikzcirclee{1pt}$};
\node[font=\tiny] at (3.5,1.2) {$b$};

\end{tikzpicture}
    \caption{$P_p$ paired with a fibered thin knot $K$ with $|\tau(K)|<g(K)$}\label{Pppairingthin}
  \end{minipage}
\end{figure}
\begin{lemma}\label{fiberedtau=g}
If $K$ is a fibered knot with $\tau(K)=\pm g(K)$, then $$\rk(\HFKhat(S^3,P_{p,1}(K),g(P_{p,1}(K))-1))=1.$$

\end{lemma}

\begin{proof}

Since $K$ is a fibered knot with $\tau(K)=\pm g(K)$, we must have $\epsilon(K)=sgn(\tau(K))$ and the essential curve component has the form described in Lemma \ref{essentialcomponent}. Since the knot is fibered, there are no other components of the immersed curve that pass through at height $-g(K)$, so the red arcs shown in figure \ref{taupos} are representative of what the immersed curve of a general fibered knot with $\tau(K)=\pm g(K)$ looks like near the bottom row of the lifted pairing diagram. As we showed in the proof of Theorem \ref{tauoftrefoil}, the knot $P_{p,1}(K)$ has one dimensional Floer homology in the top most Alexander grading, and the intersection point labelled $a$ carries this Alexander grading. Continuing with this reasoning we have that $A(b)=g(P_{p,1}(K))-1$ by Lemma \ref{Alex}, since starting from $a$, we encounter one $\delta_{w,z}$ arc before we reach the intersection point labelled $b$. Now, no other intersections between the $\beta$ curve and $\alpha(K)$ occur before we reach another $\delta_{w,z}$ arc while traversing $\beta$ up the diagram. Since the Alexander grading is weakly decreasing as we travel up the $\beta$ curve, it follows that $b$ is the unique intersection point with Alexander grading $g(P_{p,1}(K))-1$, hence $\HFKhat(S^3,P_{p,1}(K),g(P_{p,1}(K))-1)$ is one dimensional, as desired. \qedhere

\end{proof}

\begin{lemma}\label{thinmonodromy}
Suppose $K$ is a fibered thin knot with $|\tau(K)|<g(K)$. Then $$\rk(\HFKhat(S^3,P_{p,1}(K),g(P_{p,1}(K))-1))=1$$
\end{lemma}

\begin{proof}

First, recall from \cite{petkovathin} that $\CFK^{\infty}$ has a simultaneously vertically and horizontally simplified basis with repect to which it decomposes as a direct sum of a staircase summand and boxes, and all horizontal and vertical differentials have length one. In this case the immersed curve $\alpha(K)$ consists of the essential component together with figure eight components, as shown in figure \ref{Pppairingthin}. Since $|\tau(K)|<g(K)$, the essential component of the immersed curve doesn't pass through at height $\pm g(K)$ and the portion that does consists of a single figure eight component, as shown in figure \ref{Pppairingthin}. We know that the intersection point $a$ represents the sole generator with Alexander grading $g(P_{p,1}(K))$, and the generator $b$ has Alexander grading $g(P_{p,1}(K))-1$. Similar to the proof of Lemma \ref{fiberedtau=g}, we see from the diagram that any other intersection point has Alexander grading $< g(P_{p,1}(K))-1$. Therefore $b$ is the sole intersection point with Alexander grading $g(P_{p,1}(K))-1$ so $\HFKhat(S^3,P_{p,1}(K),g(P_{p,1}(K))-1)$ is one dimensional, as desired. \qedhere.

\end{proof}

\begin{proof}[Proof of Theorem \ref{monodromy}]
If $K$ is any fibered knot such that $\tau(K)=\pm g(K)$, then Lemma \ref{fiberedtau=g} implies that $P_{p,1}(K)$ has left or right veering monodromy. If $K$ is any fibered thin knot such that $|\tau(K)|<g(K)$, then Lemma \ref{thinmonodromy} implies that $P_{p,1}(K)$ has left or right veering monodromy. \qedhere

\end{proof}

Finally, we prove Proposition \ref{thinsatellite} from the introduction. 
\begin{proof}[Proof of Proposition \ref{thinsatellite}]
Suppose $K$ is a non-trivial fibered thin knot with $|\tau(K)|<g(K)$. By Theorem \ref{monodromy}, the fibered knot $P_{p,1}(K)$ has right or left veering monodromy. Therefore, by \cite{baldwinveering}*{Corollary 1.7} to show that $P_{p,1}(K)$ is not thin it is enough to show for each $p>1$ that $|\tau(P_{p,1}(K)|<g(P_{p,1}(K))$. By Corollary \ref{satellitegenus} we know that $g(P_{p,1}(K))=(p+1)g(K)+1$. 
In the case $\epsilon(K)=0$, Theorem \ref{tauoftrefoil} implies $\tau(P_{p,1}(K))=1<(p+1)g(K)+1=g(P_{p,1}(K)),$ since $g(K)>0$. In the case $\epsilon(K)=1$, Theorem \ref{tauoftrefoil} implies $|\tau(P_{p,1}(K))|\leq (p+1)|\tau(K)|+1<(p+1)g(K)+1=g(P_{p,1}(K)),$ where the strict inequality is by assumption. In the case $\epsilon(K)=-1$ we have that $-g < \tau(K)<g$. Then $|\tau(P_{p,1}(K))|=(p+1)|\tau(K)+1|<(p+1)g(K)<(p+1)g(K)+1=g(P_{p,1}(K)).$\qedhere

\end{proof}

\bibliographystyle{plain}
\bibliography{trefoilpatterns}

\begin{bibdiv}
\begin{biblist}

\bib{BLH}{article}{
      author={Baldwin, John~A.},
      author={Hedden, Matthew},
      author={Lobb, Andrew},
       title={On the functoriality of {K}hovanov-{F}loer theories},
        date={2019},
        ISSN={0001-8708},
     journal={Adv. Math.},
      volume={345},
       pages={1162\ndash 1205},
         url={https://doi.org/10.1016/j.aim.2019.01.026},
      review={\MR{3903915}},
}

\bib{baldwinveering}{misc}{
      author={Baldwin, John~A.},
      author={Ni, Yi},
      author={Sivek, Steven},
       title={Floer homology and right-veering monodromy},
   publisher={arXiv},
        date={2022},
         url={https://arxiv.org/abs/2204.04093},
}

\bib{Chen}{misc}{
      author={Chen, Wenzhao},
       title={Knot {F}loer homology of satellite knots with (1,1)-patterns},
   publisher={arXiv},
        date={2019},
         url={https://arxiv.org/abs/1912.07914},
}

\bib{combfloer}{article}{
      author={de~Silva, Vin},
      author={Robbin, Joel~W.},
      author={Salamon, Dietmar~A.},
       title={Combinatorial {F}loer homology},
        date={2014},
        ISSN={0065-9266},
     journal={Mem. Amer. Math. Soc.},
      volume={230},
      number={1080},
       pages={v+114},
      review={\MR{3205426}},
}

\bib{Heddenwhitehead}{article}{
      author={Hedden, Matthew},
       title={Knot {F}loer homology of {W}hitehead doubles},
        date={2007},
        ISSN={1465-3060},
     journal={Geom. Topol.},
      volume={11},
       pages={2277\ndash 2338},
         url={https://doi.org/10.2140/gt.2007.11.2277},
      review={\MR{2372849}},
}

\bib{fiberedsatellite}{article}{
      author={Hirasawa, Mikami},
      author={Murasugi, Kunio},
      author={Silver, Daniel~S.},
       title={When does a satellite knot fiber?},
        date={2008},
        ISSN={0018-2079},
     journal={Hiroshima Math. J.},
      volume={38},
      number={3},
       pages={411\ndash 423},
         url={http://projecteuclid.org/euclid.hmj/1233152778},
      review={\MR{2477750}},
}

\bib{Hom}{article}{
      author={Hom, Jennifer},
       title={Bordered {H}eegaard {F}loer homology and the tau-invariant of
  cable knots},
        date={2014},
        ISSN={1753-8416},
     journal={J. Topol.},
      volume={7},
      number={2},
       pages={287\ndash 326},
         url={https://doi.org/10.1112/jtopol/jtt030},
      review={\MR{3217622}},
}

\bib{HRW1}{misc}{
      author={Hanselman, Jonathan},
      author={Rasmussen, Jacob},
      author={Watson, Liam},
       title={Bordered {F}loer homology for manifolds with torus boundary via
  immersed curves},
        date={2017},
}

\bib{HRW}{article}{
      author={Hanselman, Jonathan},
      author={Rasmussen, Jacob},
      author={Watson, Liam},
       title={Heegaard {F}loer homology for manifolds with torus boundary:
  properties and examples},
        date={2022},
        ISSN={0024-6115},
     journal={Proc. Lond. Math. Soc. (3)},
      volume={125},
      number={4},
       pages={879\ndash 967},
         url={https://doi.org/10.1112/plms.12473},
      review={\MR{4500201}},
}

\bib{geography}{article}{
      author={Hedden, Matthew},
      author={Watson, Liam},
       title={On the geography and botany of knot {F}loer homology},
        date={2018},
        ISSN={1022-1824},
     journal={Selecta Math. (N.S.)},
      volume={24},
      number={2},
       pages={997\ndash 1037},
         url={https://doi.org/10.1007/s00029-017-0351-5},
      review={\MR{3782416}},
}

\bib{cabling}{misc}{
      author={Hanselman, Jonathan},
      author={Watson, Liam},
       title={Cabling in terms of immersed curves},
        date={2019},
}

\bib{surfacedecomp}{article}{
      author={Juh\'{a}sz, Andr\'{a}s},
       title={{F}loer homology and surface decompositions},
        date={2008},
        ISSN={1465-3060},
     journal={Geom. Topol.},
      volume={12},
      number={1},
       pages={299\ndash 350},
         url={https://doi.org/10.2140/gt.2008.12.299},
      review={\MR{2390347}},
}

\bib{Seifertsurface}{article}{
      author={Juhasz, Andras},
       title={Knot {F}loer homology and {S}eifert surfaces},
        date={2008},
        ISSN={1472-2747},
     journal={Algebr. Geom. Topol.},
      volume={8},
      number={1},
       pages={603\ndash 608},
         url={https://doi.org/10.2140/agt.2008.8.603},
      review={\MR{2443240}},
}

\bib{Levinemazur}{article}{
      author={Levine, Adam~Simon},
       title={Nonsurjective satellite operators and piecewise-linear
  concordance},
        date={2016},
     journal={Forum Math. Sigma},
      volume={4},
       pages={Paper No. e34, 47},
         url={https://doi.org/10.1017/fms.2016.31},
      review={\MR{3589337}},
}

\bib{LOT}{article}{
      author={Lipshitz, Robert},
      author={Ozsvath, Peter~S.},
      author={Thurston, Dylan~P.},
       title={Bordered {H}eegaard {F}loer homology},
        date={2018},
        ISSN={0065-9266},
     journal={Mem. Amer. Math. Soc.},
      volume={254},
      number={1216},
       pages={viii+279},
         url={https://doi.org/10.1090/memo/1216},
      review={\MR{3827056}},
}

\bib{quasialt}{inproceedings}{
      author={Manolescu, Ciprian},
      author={Ozsv\'{a}th, Peter},
       title={On the {K}hovanov and knot {F}loer homologies of
  quasi-alternating links},
        date={2008},
   booktitle={Proceedings of {G}\"{o}kova {G}eometry-{T}opology {C}onference
  2007},
   publisher={G\"{o}kova Geometry/Topology Conference (GGT), G\"{o}kova},
       pages={60\ndash 81},
      review={\MR{2509750}},
}

\bib{Nisutured}{article}{
      author={Ni, Yi},
       title={Sutured {H}eegaard diagrams for knots},
        date={2006},
        ISSN={1472-2747},
     journal={Algebr. Geom. Topol.},
      volume={6},
       pages={513\ndash 537},
         url={https://doi.org/10.2140/agt.2006.6.513},
      review={\MR{2220687}},
}

\bib{Nifibered}{article}{
      author={Ni, Yi},
       title={Knot {F}loer homology detects fibred knots},
        date={2007},
        ISSN={0020-9910},
     journal={Invent. Math.},
      volume={170},
      number={3},
       pages={577\ndash 608},
         url={https://doi.org/10.1007/s00222-007-0075-9},
      review={\MR{2357503}},
}

\bib{Niveering}{misc}{
      author={Ni, Yi},
       title={Exceptional surgeries on hyperbolic fibered knots},
   publisher={arXiv},
        date={2020},
         url={https://arxiv.org/abs/2007.11774},
}

\bib{OStau}{article}{
      author={Ozsv\'{a}th, Peter},
      author={Szab\'{o}, Zolt\'{a}n},
       title={Knot {F}loer homology and the four-ball genus},
        date={2003},
        ISSN={1465-3060},
     journal={Geom. Topol.},
      volume={7},
       pages={615\ndash 639},
         url={https://doi.org/10.2140/gt.2003.7.615},
      review={\MR{2026543}},
}

\bib{genusdetection}{article}{
      author={Ozsv\'{a}th, Peter},
      author={Szab\'{o}, Zolt\'{a}n},
       title={Holomorphic disks and genus bounds},
        date={2004},
        ISSN={1465-3060},
     journal={Geom. Topol.},
      volume={8},
       pages={311\ndash 334},
         url={https://doi.org/10.2140/gt.2004.8.311},
      review={\MR{2023281}},
}

\bib{OSknot}{article}{
      author={Ozsv\'{a}th, Peter},
      author={Szab\'{o}, Zolt\'{a}n},
       title={Holomorphic disks and knot invariants},
        date={2004},
        ISSN={0001-8708},
     journal={Adv. Math.},
      volume={186},
      number={1},
       pages={58\ndash 116},
         url={https://doi.org/10.1016/j.aim.2003.05.001},
      review={\MR{2065507}},
}

\bib{petkovathin}{article}{
      author={Petkova, Ina},
       title={Cables of thin knots and bordered {H}eegaard {F}loer homology},
        date={2013},
        ISSN={1663-487X},
     journal={Quantum Topol.},
      volume={4},
      number={4},
       pages={377\ndash 409},
         url={https://doi.org/10.4171/QT/43},
      review={\MR{3134023}},
}

\bib{petkovamazur}{misc}{
      author={Petkova, Ina},
      author={Wong, Biji},
       title={Twisted {M}azur pattern satellite knots and bordered {F}loer
  theory},
        date={2021},
}

\bib{Rasknot}{book}{
      author={Rasmussen, Jacob~Andrew},
       title={Floer homology and knot complements},
   publisher={ProQuest LLC, Ann Arbor, MI},
        date={2003},
        ISBN={978-0496-39374-9},
  url={http://gateway.proquest.com/openurl?url_ver=Z39.88-2004&rft_val_fmt=info:ofi/fmt:kev:mtx:dissertation&res_dat=xri:pqdiss&rft_dat=xri:pqdiss:3091665},
        note={Thesis (Ph.D.)--Harvard University},
      review={\MR{2704683}},
}

\bib{schubert}{article}{
      author={Schubert, Horst},
       title={Knoten und {V}ollringe},
        date={1953},
        ISSN={0001-5962},
     journal={Acta Math.},
      volume={90},
       pages={131\ndash 286},
         url={https://doi.org/10.1007/BF02392437},
      review={\MR{72482}},
}

\bib{Zhang}{article}{
      author={Zhang, Melissa},
       title={A rank inequality for the annular {K}hovanov homology of
  2-periodic links},
        date={2018},
        ISSN={1472-2747},
     journal={Algebr. Geom. Topol.},
      volume={18},
      number={2},
       pages={1147\ndash 1194},
         url={https://doi.org/10.2140/agt.2018.18.1147},
      review={\MR{3773751}},
}

\end{biblist}
\end{bibdiv}

\end{document}